\documentclass[11pt, a4paper]{amsart}
\usepackage[english]{babel}
\usepackage{amsmath, amsfonts, amssymb, amsthm, latexsym}
\usepackage{mathrsfs, mathtools, stmaryrd}
\usepackage{centernot}
\usepackage{bbm} 
\usepackage{mathdots}
\usepackage{comment}

\usepackage{graphicx}
\usepackage[rightcaption]{sidecap}
\usepackage{wrapfig}
\usepackage{subfig}
\usepackage[table,xcdraw]{xcolor}
\usepackage{stackengine} 
\usepackage{geometry}
\usepackage{tikz}
\usepackage{tikz-cd}

\usepackage{hyperref}
\hypersetup{
    colorlinks=true,
    linkcolor=blue,
    filecolor=magenta,      
    urlcolor=cyan,
    citecolor=red,
}

\theoremstyle{plain}
\newtheorem{theorem}{Theorem}[section]
\newtheorem{lemma}[theorem]{Lemma}
\newtheorem{proposition}[theorem]{Proposition}
\newtheorem{corollary}[theorem]{Corollary}

\theoremstyle{definition}

\theoremstyle{remark}
\newtheorem{remark}[theorem]{Remark}

\DeclareMathOperator{\id}{id}

\DeclareMathOperator{\Sym}{Sym}
\DeclareMathOperator{\Spec}{Spec}
\DeclareMathOperator{\Ai}{Ai}

\DeclareMathOperator{\SL}{SL}
\DeclareMathOperator{\Gal}{Gal}

\DeclareMathOperator{\Var}{Var}

\DeclareMathOperator{\Frob}{Frob}

\DeclareMathOperator{\csup}{c}
\DeclareMathOperator{\loc}{loc}

\NewDocumentCommand{\Gm}{e{_}}{\mathbb{G}_{\mathrm{m}\IfValueT{#1}{,#1}}}
\DeclareMathOperator{\Swan}{Swan}
\DeclareMathOperator{\geom}{geom}

\newcommand{\Qlbar}{\overline{\mathbb{Q}}_{\ell}}

\newcommand{\Qp}{\mathbb{Q}_{p}}

\newcommand{\Fp}{\mathbb{F}_{p}}

\title{On the Generalized Arithmetic Picard--Lefschetz Formula}
\author{Ping-Hsun Chuang}
\date{\today}

\begin{document}

\begin{abstract}
This paper establishes an explicit arithmetic Picard--Lefschetz formula for isolated singularities of diagonal type. We investigate proper flat families of varieties over a henselian discrete valuation ring whose fraction field has characteristic zero. When the smooth generic fibre degenerates into a special fibre with isolated diagonal singularities, we construct an explicit eigenbasis for the tame inertia action on the local cohomology and obtain formulas for the intersection pairing and the variation morphism in this basis. When the residue field is finite, we also compute the traces of geometric Frobenius on the inertia invariants in terms of Jacobi sums.

We apply this framework to the arithmetic study of symmetric powers of the Airy sheaf. Using geometric realizations associated with Airy moments, we establish specialization exact sequences for compactly supported and middle cohomology and determine the Frobenius characteristic polynomials of inertia-invariant local contributions.
\end{abstract}

\maketitle

\tableofcontents

\section{Introduction}
\label{sec:intro}

\subsection{Background and motivation}
\label{subsec:background_motivation}
The classical Picard--Lefschetz formula describes the monodromy
action on the middle cohomology of a smooth fibre near a
degeneration with an ordinary quadratic singularity.
More generally, the study of
isolated hypersurface singularities involves the Milnor fibre,
its monodromy, and the intersection pairing on vanishing cycles;
see \cite{Mil68,AGV1988,Dimca1992}.

For diagonal (Brieskorn--Pham) singularities
\[
 f(z_0,\ldots,z_n)=z_0^{k_0}+\cdots+z_n^{k_n},
\]
Pham \cite{Pham1965} obtained explicit generalized Picard--Lefschetz
formulas for the variation morphism and the intersection form.
The Sebastiani--Thom theorem \cite{ST1971} describes the monodromy
of sums of functions in disjoint variables. Related intersection
forms for quasi-homogeneous singularities were studied by Steenbrink
\cite{MR453735}.

In the arithmetic setting, nearby and vanishing cycles
relate the \'etale cohomology of the generic and special
fibres of a family over a henselian discrete valuation ring.
They also carry a natural action of the local Galois group.
Grothendieck's vanishing-cycle formalism and
Deligne's arithmetic Picard--Lefschetz formula for ordinary quadratic
singularities are developed in \cite{SGA7_II}.
Illusie subsequently gave a purely algebraic proof of the
Picard--Lefschetz formula and treated certain homogeneous
singularities \cite{Ill00}. His later work \cite{Ill03}
established concentration results for vanishing cycles
and studied the tame variation morphism for isolated
singularities.

The aim of this paper is to make the arithmetic formulas explicit
for isolated singularities of diagonal type. Besides the tame inertia
action, we describe the local intersection pairing and variation map,
and, over finite residue fields, compute Frobenius traces on inertia
invariants in terms of Jacobi sums. The classical diagonal formulas
provide the eigenbasis and normalization, while nearby cycles and
local duality supply their arithmetic counterparts.
The Frobenius trace formulas for diagonal hypersurfaces go back
to Weil \cite{Wei49}. Here we derive the characterwise formulas
using exponential cohomology and Kummer decomposition.

A principal motivation comes from Airy moments. The $\ell$-adic Airy
sheaf and its monodromy were studied by Katz in connection with
exponential sums \cite{Kat87,Kat90}. The Hodge-theoretic structures
of symmetric-power Airy moments were studied by Sabbah and Yu
\cite{SY23}, who computed their irregular Hodge numbers using
exponential mixed Hodge structures. Following their geometric
construction, let
\[
g_k(y)=\sum_{i=1}^k\left(\frac{y_i^3}{3}-y_i\right),
\qquad
A'=\{g_k(y)=0\}\subset\mathbb A_y^k,
\]
and consider 
\begin{equation}\label{eq:A''_def_eq}
A''=\{(z,y)\in\Gm_z\times\mathbb A_y^k:
z^3=g_k(y)\}.
\end{equation}
The projection
\[
A''\longrightarrow\mathbb A_y^k\setminus A',
\qquad (z,y)\longmapsto y,
\]
is a finite \'etale cubic cover.

The actions of $S_k\times\mu_2$ and the deck group $\mu_3$
allow us to isolate the cohomological realizations relevant
to Airy moments. 
For $p\ge5$, the projective compactifications
of $A'$ and $A''$ may have isolated singularities
of types $A_1$ and $A_2$, respectively.
The arithmetic Picard--Lefschetz formulas allow us
to determine their local contributions to the
specialization sequences.

\subsection{Main results}
\label{subsec:main_results}
We first summarize the generalized arithmetic Picard--Lefschetz
formula proved in Theorem~\ref{thm:main_gen_PL}. The explicit
variation coefficients and Frobenius character conventions are given
there.

\begin{theorem}[Arithmetic Picard--Lefschetz for diagonal singularities]
\label{thm:intro_main}
Let $f:X\to S=\Spec R$ be proper and flat of relative dimension
$n\ge1$, with smooth geometric generic fibre, where $R$ is a
henselian discrete valuation ring with fraction field of
characteristic zero, uniformizer $\pi$, and residue characteristic
$p$. Fix $\ell\ne p$. Assume that $f$ is smooth away from a
finite set $\Sigma\subset X_{\bar s}$ and that near each
$x\in\Sigma$ it admits an \'etale-local equation over $R$
\[
 \sum_{i=0}^n a_{i,x}z_i^{k_{i,x}}+u_x\pi^{v_x}=0,
\]
where $a_{i,x},u_x\in R^\times$, the integers
$k_{i,x}\ge2$ are invertible in $R$, and $v_x\ge1$.
Put
\[
\begin{gathered}
 V_x=H^n_{\{x\}}(X_{\bar s},R\Psi\Qlbar),\qquad
 W_x=(R^n\Phi\Qlbar)_x,\\
 \Xi_x=\prod_{i=0}^n(\mu_{k_{i,x}}\setminus\{1\}),\qquad
 K_x=\prod_{i=0}^n k_{i,x},\qquad
 \Lambda_{\boldsymbol\zeta,x}=\prod_{i=0}^n\zeta_i.
\end{gathered}
\]
Then the following hold.
\begin{enumerate}
\item Each $V_x$ has a basis
$\{e_{\boldsymbol\zeta,x}\}_{\boldsymbol\zeta\in\Xi_x}$
on which a chosen tame generator acts by
\[
 \sigma_{\mathrm{top}}e_{\boldsymbol\zeta,x}
 =\Lambda_{\boldsymbol\zeta,x}^{v_x}e_{\boldsymbol\zeta,x}.
\]
With the normalization
\eqref{eq:local_intersection_definition},
the local intersection form is given by
\[
 (e_{\boldsymbol\zeta,x},e_{\boldsymbol\xi,x})_x
 =\begin{cases}
 \displaystyle (-1)^{n(n+1)/2}K_x
 \frac{\Lambda_{\boldsymbol\zeta,x}-1}
 {\prod_{i=0}^n(\zeta_i-1)},
 &\boldsymbol\xi=\boldsymbol\zeta^{-1},\\[5pt]
 0,&\text{otherwise}.
 \end{cases}
\]
\item The specialization maps fit into the exact sequence
\[
\begin{gathered}
0\longrightarrow H^n(X_{\bar s},\Qlbar)
\longrightarrow H^n(X_{\bar\eta},\Qlbar)
\longrightarrow\bigoplus_{x\in\Sigma}W_x\\
\longrightarrow H^{n+1}(X_{\bar s},\Qlbar)
\longrightarrow H^{n+1}(X_{\bar\eta},\Qlbar)
\longrightarrow0.
\end{gathered}
\]
The specialization maps are isomorphisms in all other degrees.
The inertia action
on $H^n(X_{\bar\eta},\Qlbar)$ is tame and is determined explicitly
by the local variation maps and the intersection pairing.
\item If the residue field is $\mathbb F_q$, the action of
geometric Frobenius on $V_x^I$, for each
$\mathbb F_q$-rational $x\in\Sigma$, has the Jacobi-sum
trace formula of Theorem~\ref{thm:Frob_trace_on_inertia_inv}.
\end{enumerate}
\end{theorem}

The local calculation uses a compactification of the
diagonal hypersurface with relative simple normal crossings
boundary, together with the classical variation formulas.
Local duality then gives the arithmetic inertia formula.
For the Frobenius trace,
we use an exponential-cohomology realization of diagonal hypersurfaces,
together with Kummer decomposition and Gauss--Jacobi sums.

We next apply these results to the geometric realizations of Airy
moments. Let $G=S_k\times\mu_2$ act on $A''$ defined in
\eqref{eq:A''_def_eq}, where $S_k$ permutes the $y_i$
and $\mu_2$ acts by $(z,y)\mapsto(-z,-y)$.
Put $\chi=\operatorname{sgn}\boxtimes\mathbf1$.
For $F\in\{\Qp,\Fp\}$, write
\[
M^\flat_{k,F}
=H_{\mathrm{mid}}^k(A''_{\overline F},\Qlbar)^\chi.
\]
The connection with the non-classical Airy realization $M''_k$
is given by Proposition~\ref{prop:airy_nonclassical_middle_identification}:
\[
M''_k\simeq(\mathbb E_3\otimes M_k^\flat)^{\mu_3},
\]
where $\mathbb E_3$ is the cubic exponential cohomology introduced
in \S\ref{subsec:airy_moments_motive}, and the invariants are
taken under the diagonal $\mu_3$-action.
The specialization of $M_k^\flat$ is described
by the following theorem.

\begin{theorem}[Specialization of the middle geometric Airy factor]
\label{thm:intro_airy}
Let $p\ge5$, $\ell\ne p$, and $k\ge2$. Put
\[
 \Theta_p^*=\{a\in\mathbb Z_{\ge1}:
 ap\le k,\ a\equiv k\pmod2\}.
\]
There is a canonical short exact sequence of local Galois
representations
\[
 0\longrightarrow M^\flat_{k,\Fp}
 \longrightarrow M^\flat_{k,\Qp}
 \longrightarrow\bigoplus_{a\in\Theta_p^*}\mathscr W_a
 \longrightarrow0,
\]
where each $\mathscr W_a$ is the rank-two sign-isotypic
vanishing-cycle contribution of an $A_2$-singularity orbit.
Taking inertia invariants gives another short exact sequence,
\[
 0\longrightarrow M^\flat_{k,\Fp}
 \longrightarrow (M^\flat_{k,\Qp})^{I_p}
 \longrightarrow\bigoplus_{a\in\Theta_p^*}\mathscr W_a^{I_p}
 \longrightarrow0.
\]
The $\mu_3$-trivial part of $M^\flat_{k,F}$ vanishes.
The inertia eigenvalues on $\mathscr W_a$ follow from
Theorem~\ref{thm:local_diagonal_main}, while the Frobenius
characteristic polynomials of $\mathscr W_a^{I_p}$ are given
explicitly in Theorem~\ref{thm:airy_local_frobenius_factors}.
Consequently, the Frobenius characteristic polynomial of
$(M^\flat_{k,\Qp})^{I_p}$ factors into the corresponding
special-fibre and local contributions.
\end{theorem}

The proof begins with the specialization sequence for compactly
supported cohomology, where the local intersection pairing
establishes surjectivity. The smooth Fermat boundary then allows
us to pass to middle cohomology. A weight argument gives the
exact sequence of inertia invariants, while the local pairing
and Gauss--Jacobi formulas determine the Frobenius factors. 

The specialization results for $A'$, including the
characteristic-two case, are also established in
\S\ref{subsec:airy_arithmetic_revised}.

\subsection{Organization of the paper}
\label{subsec:organization}
Section~\ref{sec:classical_PL} recalls the classical diagonal
Picard--Lefschetz formulas and fixes the sign conventions.
Section~\ref{sec:diag_singularity} establishes their arithmetic
counterparts, including the local Frobenius trace formula.
Section~\ref{sec:application} studies Airy sheaves and the
cohomological realizations of Airy moments, and applies the formula
to their local Galois representations.

\section{Classical Picard--Lefschetz formulas for diagonal singularities}
\label{sec:classical_PL}

We briefly recall the classical formulas needed in the arithmetic setting.
For a holomorphic function $f:(\mathbb C^{n+1},0)\to(\mathbb C,0)$
with an isolated critical point, choose $0<|\varepsilon|\ll\delta\ll1$
and let
\[
 F=f^{-1}(\varepsilon)\cap\{\textstyle\sum_i|z_i|^2\leq\delta^2\}
\]
be its Milnor fibre. It has the homotopy type of a bouquet of
$\mu$ spheres of dimension $n$, where $\mu$ is the Milnor number
\cite{Mil68}. The geometric monodromy is represented by a
diffeomorphism of $F$ fixed near its boundary.
All results in this section are known: the generalized Picard--Lefschetz
formulas for diagonal (Brieskorn--Pham) singularities were established by
Pham \cite[Part~I]{Pham1965}. We use the Thom--Sebastiani theorem
\cite{ST1971} to express them in terms of one-variable factors; see also
\cite{Mil68,Dimca1992} for the general theory. Our purpose here is only to
fix the eigenvectors and sign conventions used below.
All (co)homology groups in this section have complex coefficients. We follow the conventions for the duality pairing,
the variation morphism, and the intersection form
in \cite[\S1.3--1.6 and \S1.9]{MR425167}.

Let $T$ denote the geometric monodromy.
We use the identifications provided by Poincar\'e--Lefschetz duality
between the reduced middle homology of $F$ and the compactly supported
cohomology of the interior $\mathring F$. Write
\[
 j:\widetilde H^n_{\csup}(\mathring F)\longrightarrow
   \widetilde H^n(F),\qquad
 \Var:\widetilde H^n(F)\longrightarrow
   \widetilde H^n_{\csup}(\mathring F)
\]
for the natural map and the variation morphism, respectively. Denote by
$\langle\ ,\ \rangle$ the duality pairing. We normalize the intersection
form on $\widetilde H^n_{\csup}(\mathring F)$ by
\begin{equation}\label{eq:classical_intersection_normalization}
 (a,b)=-(-1)^{n(n-1)/2}\langle j(a),b\rangle.
\end{equation}
The two monodromy operators and the variation satisfy
\[
 T|_{\widetilde H^n_c(\mathring F)}=\id+\Var\circ j,
 \qquad
 T|_{\widetilde H^n(F)}=\id+j\circ\Var.
\]
These identities, and the composition rule for variation, will be used
when the arithmetic monodromy is identified with $T^v$ later. We use the
normalization \eqref{eq:classical_intersection_normalization} throughout.

\subsection{The cyclic case}
\label{subsec:cyclic_case}

For $f(z)=z^k$, the Milnor fibre is the set $F=\mu_k$. Fix
$\omega=\exp(2\pi i/k)$ and the basis
\[
 \delta_j=[\omega^{k-j-1}]-[\omega^{k-j}]
 \quad (1\leq j\leq k-1)
\]
of $\widetilde H_0(F)$. For each $\zeta\in\mu_k\setminus\{1\}$, put
\[
 e_\zeta=\sum_{j=1}^{k-1}
       \frac{1-\zeta^j}{1-\zeta}\,\delta_j,
\]
and let $e_\zeta^*$ be the dual basis. 
The following elementary identities
record the one-variable case of the classical formulas of
\cite{Pham1965}, in our normalization:
\begin{equation}\label{eq:cyclic_classical_formulas}
 \begin{gathered}
 T(e_\zeta)=\zeta e_\zeta,\qquad
 \Var(e_\zeta^*)=\frac{\zeta^{-1}(\zeta-1)}{k}\,e_{\zeta^{-1}},\\
 (e_\zeta,e_\xi)=
 \begin{cases}
 k & \text{if }\xi=\zeta^{-1},\\
 0 & \text{otherwise}.
 \end{cases}
 \end{gathered}
\end{equation}
Here, the variation is
$\Var:\widetilde H^0(F)\to\widetilde H_0(F)$.
Writing $\{\delta_j^*\}$ for the dual basis, one has
$\Var(\delta_j^*)=\sum_{i=j}^{k-1}\delta_i$;
the pairing on the $\delta_i$ is the $A_{k-1}$ Cartan matrix
($2$ on the diagonal, $-1$ on adjacent diagonals).
These elementary expressions yield both remaining formulas in
\eqref{eq:cyclic_classical_formulas}.

\subsection{Diagonal singularities and Thom--Sebastiani}
\label{subsec:diag_singularity_cl}

Consider
\[
 f(z_0,\ldots,z_n)=z_0^{k_0}+\cdots+z_n^{k_n},
 \qquad k_i\geq2,
\]
and let $F$ be its Milnor fibre. Set
\[
 \Xi=\prod_{i=0}^n(\mu_{k_i}\setminus\{1\}),\qquad
 K=\prod_{i=0}^n k_i.
\]
By Thom--Sebastiani \cite{ST1971} and Poincar\'e--Lefschetz duality,
the compactly supported middle cohomology of $\mathring F$ is identified with the tensor product of the reduced
degree-zero homologies of the cyclic Milnor fibres. In the preceding
normalization, it has a basis
\[
 e_{\boldsymbol\zeta}
   =e_{\zeta_0}^{(0)}\otimes\cdots\otimes e_{\zeta_n}^{(n)},
 \qquad \boldsymbol\zeta=(\zeta_0,\ldots,\zeta_n)\in\Xi,
\]
where $e_{\zeta_i}^{(i)}$ is as in the cyclic case. Its dimension is the
Milnor number $\prod_{i=0}^n(k_i-1)$, and
\[
 T(e_{\boldsymbol\zeta})=\Lambda_{\boldsymbol\zeta}
   e_{\boldsymbol\zeta},\qquad
 \Lambda_{\boldsymbol\zeta}=\prod_{i=0}^n\zeta_i.
\]

The following formulas are obtained from the classical
diagonal formulas of Pham \cite[Part~I]{Pham1965}
and the tensor-product formula for variation
\cite[\S1.9]{MR425167}, expressed in the eigenbasis above. We
record them without proof, since only these explicit expressions will
be used in the arithmetic setting.

\begin{proposition}[Classical variation formula]\label{prop:diag_variation}
Let $\{e_{\boldsymbol\zeta}^*\}_{\boldsymbol\zeta\in\Xi}$ denote the
basis dual to $\{e_{\boldsymbol\zeta}\}_{\boldsymbol\zeta\in\Xi}$. Then
\[
 \Var(e_{\boldsymbol\zeta}^*)
 =\frac{\Lambda_{\boldsymbol\zeta}^{-1}}{K}
   \left(\prod_{i=0}^n(\zeta_i-1)\right)
   e_{\boldsymbol\zeta^{-1}},
\]
where $\boldsymbol\zeta^{-1}=(\zeta_0^{-1},\ldots,\zeta_n^{-1})$.
\end{proposition}

\begin{proposition}[Classical intersection pairing]\label{prop:cl_diag_pairing}
For $\boldsymbol\zeta,\boldsymbol\xi\in\Xi$, the intersection form
normalized by \eqref{eq:classical_intersection_normalization} satisfies
\[
 (e_{\boldsymbol\zeta},e_{\boldsymbol\xi})=
 \begin{cases}
 \displaystyle (-1)^{n(n+1)/2}
   \frac{\Lambda_{\boldsymbol\zeta}-1}
        {\prod_{i=0}^n(\zeta_i-1)}\,K
   & \text{if }\boldsymbol\xi=\boldsymbol\zeta^{-1},\\[6pt]
 0 & \text{otherwise}.
 \end{cases}
\]
\end{proposition}

\section{Arithmetic Picard--Lefschetz formula for diagonal singularities}
\label{sec:diag_singularity}

We first recall the global specialization and variation formalism, then
compute the local monodromy and variation for diagonal singularities.
We next determine the Frobenius action on inertia invariants, using the
full diagonal symmetry and an equivariant exponential-cohomology
realization, and conclude with the global formula.

\subsection{Global reduction to local diagonal singularities}
\label{subsec:global_reduction}

We recall the specialization sequence and the local variation formalism of
\cite[Exp.~XIII]{SGA7_II}. Let $S=\Spec R$ be a henselian trait with fraction
field of characteristic zero, generic point $\eta$,
and closed point $s$, and residue
characteristic $p\geq 0$. Fix a prime $\ell$ invertible on $S$. Let
$f:X\to S$ be a proper flat morphism of relative dimension $n\geq 1$,
with smooth generic fibre. Assume that $f$ is smooth outside a finite set
$\Sigma\subset X_{\overline{s}}$ of geometric points, and that, locally
at each $x\in\Sigma$, the family is given by an equation
\begin{equation}\label{eq:local_model_eq}
  a_0z_0^{k_0}+\cdots+a_nz_n^{k_n}+u\pi^v=0,
\end{equation}
where $\pi$ is a uniformizer, $a_i,u\in R^\times$, $v\geq 1$, and
$k_i\geq 2$ are invertible in $R$. These data may depend on $x$.

Write $R\Psi=R\Psi(\Qlbar)$ and $R\Phi=R\Phi(\Qlbar)$
for the nearby-cycle and vanishing-cycle complexes of $f$,
respectively. The vanishing-cycle complex is supported on
$\Sigma$. Since $f$ is flat and locally a complete intersection
of relative dimension $n$, and is smooth along the special
fibre outside $\Sigma$, the concentration theorem of Illusie
\cite[\S 2.10]{Ill03} gives
\[
(R^q\Phi)_x=0
\qquad(q\neq n,\ x\in\Sigma).
\]
Thus $R\Phi$ is concentrated in degree $n$ and has finite
support.
Applying proper base change to the distinguished triangle
\[
 \Qlbar\longrightarrow R\Psi\longrightarrow R\Phi
 \xrightarrow{+1}
\]
gives the standard specialization sequence.

\begin{proposition}[Specialization exact sequence]
\label{prop:specialization_exact_sequence}
Under the preceding assumptions, there is an exact sequence of
$\Gal(\overline\eta/\eta)$-representations
\begin{equation}\label{eq:long_exact_seq_general}
\begin{aligned}
0\longrightarrow H^n(X_{\overline{s}},\Qlbar)
 &\xrightarrow{\,sp\,} H^n(X_{\overline\eta},\Qlbar)
 \xrightarrow{\,can\,} \bigoplus_{x\in\Sigma}(R^n\Phi)_x\\
 &\xrightarrow{\,\partial\,}H^{n+1}(X_{\overline{s}},\Qlbar)
 \xrightarrow{\,sp\,}H^{n+1}(X_{\overline\eta},\Qlbar)
 \longrightarrow 0.
\end{aligned}
\end{equation}
Furthermore, $sp:H^i(X_{\overline{s}},\Qlbar)\to
H^i(X_{\overline\eta},\Qlbar)$ is an isomorphism for $i\neq n,n+1$.
\end{proposition}

Let $I\subset\Gal(\overline\eta/\eta)$ denote the inertia subgroup,
and write $I_t$ for its tame quotient. Fix a topological generator $\sigma_{\mathrm{top}}$ of $I_t$.
For $\sigma\in I$, define $\nu_\ell(\sigma)\in\mathbb Z_\ell$
by projecting the image of $\sigma$ in the procyclic group $I_t$
to its $\mathbb Z_\ell$-factor, normalized so that
$\nu_\ell(\sigma_{\mathrm{top}})=1$.
The global variation is expressed in terms of
the local variation maps
\[
 \Var_x(\sigma):(R^n\Phi)_x\longrightarrow
 H^n_{\{x\}}(X_{\overline{s}},R\Psi),\qquad \sigma\in I,
\]
constructed in \cite[Exp.~XIII, (2.4.5.1)]{SGA7_II}. More precisely,
\cite[Exp.~XIII, (2.4.6.2)]{SGA7_II} gives the commutative diagram
\begin{equation}\label{diag:var_decomposition}
\begin{tikzcd}
H^n(X_{\overline\eta},\Qlbar)
 \arrow[rr,"\sigma-\id"] \arrow[d,"\wr"']
 && H^n(X_{\overline\eta},\Qlbar) \\
H^n(X_{\overline{s}},R\Psi)
 \arrow[d,"can"']
 && H^n(X_{\overline{s}},R\Psi)\arrow[u,"\wr"'] \\
\displaystyle\bigoplus_{x\in\Sigma}(R^n\Phi)_x
 \arrow[rr,"\bigoplus_x\Var_x(\sigma)"']
 && \displaystyle\bigoplus_{x\in\Sigma}
 H^n_{\{x\}}(X_{\overline{s}},R\Psi)\arrow[u,"q"]
\end{tikzcd}
\end{equation}
where $q$ is the map forgetting supports. Equivalently,
\[
 \sigma-\id=\sum_{x\in\Sigma}
 q_x\circ\Var_x(\sigma)\circ can_x
 \quad\text{on }H^n(X_{\overline\eta},\Qlbar).
\]

We use the following local duality statement
(cf.\ \cite[Exp.~XIII, \S2.4]{SGA7_II}).
\begin{proposition}[Local Poincar\'e duality]
\label{prop:local_duality}
For $x\in\Sigma$, set
\[
 V_x=H^n_{\{x\}}(X_{\overline s},R\Psi),\qquad
 W_x=(R^n\Phi)_x.
\]
The natural map $(R^n\Psi)_x\to W_x$ is an isomorphism, and
there is a perfect pairing, equivariant under the stabilizer of $x$ in $\Gal(\overline\eta/\eta)$,
\begin{equation}\label{eq:local_duality_pairing}
 \langle\ ,\ \rangle_{\mathrm{loc},x}:
 W_x\times V_x\longrightarrow\Qlbar(-n).
\end{equation}
In particular, $W_x\simeq V_x^\vee(-n)$.
\end{proposition}
\begin{proof}
The stalk comparison follows from the triangle defining $R\Phi$
because $n\geq1$. The perfect pairing is local Poincar\'e
duality for nearby cycles.
\end{proof}

Let $j_x:V_x\to W_x$ be the natural map which first forgets
supports and then uses $(R^n\Psi)_x\simeq W_x$.
Choose once and for all $0\ne\tau\in\Qlbar(-n)$, and write
$[b]_\tau$ for the scalar determined by $b=[b]_\tau\tau$.
The \emph{local intersection form} is the scalar-valued pairing
\begin{equation}\label{eq:local_intersection_definition}
 (a,b)_x:=-(-1)^{n(n-1)/2}
 [\langle j_x(a),b\rangle_{\mathrm{loc},x}]_\tau,
 \qquad a,b\in V_x.
\end{equation}
Here $\langle\ ,\ \rangle_{\mathrm{loc},x}$ retains its
$\Qlbar(-n)$-valued Galois-equivariant meaning; choosing $\tau$
merely expresses the intersection form and the later dual-basis
formulas as scalars. We similarly write
$[\langle\alpha,\beta\rangle_X]_\tau$ for the scalar
coordinate of the global Poincar\'e pairing on the smooth proper
generic fibre. 
By the compatibility of local and global duality,
the maps $can_x$ and $q_x$ are adjoint:
\[
\langle can_x(\alpha),v\rangle_{\mathrm{loc},x}
=\langle\alpha,q_x(v)\rangle_X
\]
for $\alpha\in H^n(X_{\overline\eta},\Qlbar)$ and
$v\in V_x$.

\subsection{Local diagonal singularities and compactification}
\label{subsec:local_diagonal}

Let $R$ be a henselian discrete valuation ring with fraction field $K$,
uniformizer $\pi$, and residue characteristic $p\geq0$. Fix a prime
$\ell\neq p$. We consider the diagonal family
\begin{equation}\label{eq:local_equation}
 X_{\loc}=\left\{\sum_{i=0}^n a_i x_i^{k_i}+u\pi^v=0\right\}
 \subseteq\mathbb A_S^{n+1},\qquad S=\Spec R,
\end{equation}
where $a_i,u\in R^\times$, $k_i\geq2$ are invertible in $R$, and $v\geq1$.
Its generic fibre is smooth, whereas the special fibre has an isolated
singularity at the origin. We first establish a compactification property
that allows us to relate nearby cycles to compactly supported cohomology.

\begin{proposition}[Compactification of diagonal hypersurfaces]
\label{prop:diagonal_compactification}
Let $m=\operatorname{lcm}(k_0,\ldots,k_n)$, and let $A$ be a
$\mathbb Z[1/m]$-algebra. For $a_i\in A^\times$ and $c\in A$,
consider
\[
 Z=\bigl\{\textstyle\sum_{i=0}^n a_ix_i^{k_i}+c=0\bigr\}
 \subset\mathbb A_A^{n+1}.
\]
There is a proper toric $A$-scheme $P$, smooth over $A$ and containing
$\mathbb A_A^{n+1}$ as a toric open subset, such that the closure
$\overline Z\subset P$ is smooth over $A$ in a neighbourhood
of its boundary $D=\overline Z\setminus Z$, and $D$ is a relative
simple normal crossings divisor on $\overline Z$ there. The toric
construction can be chosen independently of $a_i$ and $c$, hence
commutes with arbitrary base change on $A$.
\end{proposition}

\begin{proof}
Let $\Delta=\operatorname{Conv}(0,k_0e_0,\ldots,k_ne_n)$.
Choose a regular projective refinement of its normal fan
which leaves the first orthant unchanged
(cf.\ \cite[\S 2.6]{Fulton1993}).
The resulting smooth proper toric $A$-scheme $P$
contains $\mathbb A_A^{n+1}$ as an open subscheme.

On each toric stratum at infinity, the defining
equation is given by a face polynomial
\[
f_\gamma=\sum_{k_ie_i\in\gamma}a_ix_i^{k_i},
\]
where $\gamma$ is a face of $\Delta$ not containing $0$.
Since
\[
\frac{\partial f_\gamma}{\partial x_i}
=k_i a_ix_i^{k_i -1}
\]
is invertible on the torus for every term occurring
in $f_\gamma$, the face polynomials are non-degenerate
over $A$.
The standard toric non-degeneracy criterion then
shows that the closure $\overline Z$ is smooth near
the boundary and meets the toric strata transversely.
Hence $D=\overline Z\setminus Z$ is a relative
simple normal crossings divisor.
\end{proof}

\begin{remark}
The same relative SNC compactification property also holds
for hypersurfaces of ADE type in sufficiently good residue
characteristic. This can be verified case by case by choosing
suitable toric compactifications and checking the corresponding
affine charts.
\end{remark}

\begin{lemma}[Nearby cycles with compact supports]
\label{lem:nearby_cycles_with_compact_supports}
For a diagonal family $Z$ as above over a henselian trait, the natural
comparison map induces an isomorphism
\[
 R\Gamma_c(Z_{\overline s},R\Psi\Qlbar)
 \xrightarrow{\ \sim\ }
 R\Gamma_c(Z_{\overline\eta},\Qlbar).
\]
\end{lemma}
\begin{proof}
Proposition~\ref{prop:diagonal_compactification} provides a compactification
which is smooth with relative simple normal crossings boundary near infinity.
The result follows from \cite[Exp.~XIII, Prop.~2.1.9]{SGA7_II}.
\end{proof}

\begin{theorem}[Local comparison]\label{thm:local_cohomology_comparison}
Assume $n\geq1$. For the local model \eqref{eq:local_equation}, there is a natural Galois-equivariant isomorphism
\[
 H^n_{\{0\}}(X_{\loc,\overline s},R\Psi\Qlbar)
 \simeq H_c^n(X_{\loc,\overline\eta},\Qlbar).
\]
\end{theorem}
\begin{proof}
Put $m=\operatorname{lcm}(k_0,\ldots,k_n)$ and $w_i=m/k_i$.
The special fibre is a weighted homogeneous cone.
The cone argument of
\cite[Exp.~XV, Proposition~2.1.2]{SGA7_II}
applies to the constant sheaf on this cone, by passing
to a homogeneous cone via a finite diagonal covering.
Since $R\Phi\Qlbar$ is supported at the origin,
the vanishing-cycle triangle then gives
\[
H^n_{\{0\}}(X_{\loc,\overline s},R\Psi\Qlbar)
\simeq H_c^n(X_{\loc,\overline s},R\Psi\Qlbar).
\]
The latter is isomorphic to $H_c^n(X_{\loc,\overline\eta},\Qlbar)$
by Lemma~\ref{lem:nearby_cycles_with_compact_supports}.
\end{proof}

Set
\[
 V=H^n_{\{0\}}(X_{\loc,\overline s},R\Psi\Qlbar),\qquad
 \Xi=\prod_{i=0}^n(\mu_{k_i}\setminus\{1\}),\qquad
 K_0=\prod_{i=0}^n k_i,
\]
and write $\Lambda_{\boldsymbol\zeta}=\prod_i\zeta_i$ for
$\boldsymbol\zeta\in\Xi$.

\begin{theorem}[Tame monodromy and intersection pairing]
\label{thm:local_diagonal_main}
The inertia action on $V$ is tame, and
$\dim V=\prod_{i=0}^n(k_i-1)$. For the generator
$\sigma_{\mathrm{top}}$ fixed above, there is a basis
$\{e_{\boldsymbol\zeta}\}_{\boldsymbol\zeta\in\Xi}$ such that
\[
 \sigma_{\mathrm{top}}(e_{\boldsymbol\zeta})
 =\Lambda_{\boldsymbol\zeta}^{v}e_{\boldsymbol\zeta}.
\]
The local intersection form defined in
\eqref{eq:local_intersection_definition} satisfies
\[
 (e_{\boldsymbol\zeta},e_{\boldsymbol\xi})_0
 =\begin{cases}
 \displaystyle (-1)^{n(n+1)/2}
 \frac{\Lambda_{\boldsymbol\zeta}-1}
 {\prod_{i=0}^n(\zeta_i-1)}\,K_0,
   &\boldsymbol\xi=\boldsymbol\zeta^{-1},\\[6pt]
 0,&\text{otherwise}.
 \end{cases}
\]
\end{theorem}

\begin{proof}
By Theorem~\ref{thm:local_cohomology_comparison}, we may work with
$H_c^n(X_{\loc,\overline\eta},\Qlbar)$. Put
$m=\operatorname{lcm}(k_0,\ldots,k_n)$ and $w_i=m/k_i$.
After adjoining $\varpi=\pi^{1/m}$, the substitution
$x_i=\varpi^{vw_i}y_i$ identifies the geometric generic fibre with
\[
 Y:\quad \sum_i a_iy_i^{k_i}+u=0.
\]
This hypersurface has smooth reduction and, by
Proposition~\ref{prop:diagonal_compactification}, a relative SNC
compactification. Its compactly supported cohomology therefore has
trivial inertia action over $K(\varpi)$. Since $m$ is invertible in $R$,
the inertia action on $V$ is tame.

To identify this action, consider the universal diagonal family
\[
 g:\mathbb A_B^{n+1}\longrightarrow\mathbb A_B^1,
 \qquad g(x)=-U^{-1}\sum_i A_ix_i^{k_i},
 \quad B=\Spec\mathbb Z[1/(m\ell),A_0^{\pm1},\ldots,
 A_n^{\pm1},U^{\pm1}].
\]
The local coefficients are obtained by specializing
$A_i\mapsto a_i$ and $U\mapsto u$.
The relative SNC compactification shows that
$R^ng_!\Qlbar|_{\mathbb G_{m,B}}$ is lisse and compatible with
specialization. 
Over the Kummer cover $t=s^m$ of $\mathbb G_{m,B}$,
the substitution $x_i=s^{w_i}y_i$ identifies the
family with the pullback of its fibre at $t=1$.

Thus its tame monodromy is the corresponding diagonal Kummer
descent action. Over $\mathbb C$, this is the classical Milnor
monodromy $T$. Since our local model is obtained by $t=\pi^v$,
$\sigma_{\mathrm{top}}$ acts as $T^v$ under this comparison
(with compatible choices of generators). The dimension and
eigenvalues now follow from \S\ref{subsec:diag_singularity_cl}.

The comparison is compatible with local duality and the
forgetting-support map $j_0$; see
\cite[Exp.~XIII, \S\S1.4 and 2.4]{SGA7_II}.
Hence, with the fixed scalar normalization, the local
intersection form agrees with the classical form of
Proposition~\ref{prop:cl_diag_pairing}.

\end{proof}

Let $W=(R^n\Phi\Qlbar)_0$. Via local duality in Proposition~\ref{prop:local_duality}, let
$\{e_{\boldsymbol\zeta}^*\}_{\boldsymbol\zeta\in\Xi}$ be the basis
of $W$ characterized by
\[
 \langle e_{\boldsymbol\zeta}^*,e_{\boldsymbol\xi}\rangle_{\mathrm{loc},0}
 =\delta_{\boldsymbol\zeta,\boldsymbol\xi}\,\tau.
\]

\begin{proposition}[Variation for a tame generator]
\label{prop:local_var_formula_sigma_top}
For $\boldsymbol\zeta\in\Xi$ set
$\Lambda=\Lambda_{\boldsymbol\zeta}$ and
\[
 C_{\boldsymbol\zeta}(\sigma_{\mathrm{top}})
 :=\frac1{K_0}\prod_{i=0}^n(\zeta_i-1)\cdot
 \begin{cases}
  \displaystyle\frac{\Lambda^{-v}-1}{1-\Lambda},&\Lambda\ne1,\\[5pt]
  v,&\Lambda=1.
 \end{cases}
\]
Then the local variation is
\[
 \Var(\sigma_{\mathrm{top}})(e_{\boldsymbol\zeta}^*)
 =C_{\boldsymbol\zeta}(\sigma_{\mathrm{top}})
 e_{\boldsymbol\zeta^{-1}}.
\]
\end{proposition}
\begin{proof}
Under the comparison in Theorem~\ref{thm:local_diagonal_main},
write $\Var_{\mathrm{cl}}=\Var(T)$ for the classical
variation map corresponding to one positive turn around the origin.
We have $\sigma_{\mathrm{top}}=T^v$ on the compared cohomology.
The composition identity for variation
\cite[Exp.~XIII, (1.4.3.4)]{SGA7_II} gives
\[
 \Var(T^v)=\sum_{j=0}^{v-1}T^j\circ\Var_{\mathrm{cl}}.
\]
By Proposition~\ref{prop:diag_variation},
\[
 \Var_{\mathrm{cl}}(e_{\boldsymbol\zeta}^*)
 =\frac{\Lambda^{-1}}{K_0}
  \prod_i(\zeta_i-1)\,e_{\boldsymbol\zeta^{-1}},
 \qquad
 T(e_{\boldsymbol\zeta^{-1}})
 =\Lambda^{-1}e_{\boldsymbol\zeta^{-1}}.
\]
The coefficient is thus
$K_0^{-1}\prod_i(\zeta_i-1)\sum_{j=1}^v\Lambda^{-j}$,
which evaluates to the two cases above.
\end{proof}

\begin{proposition}[Local variation formula]\label{prop:local_variation_formula}
For $\boldsymbol\zeta\in\Xi$, let
$\kappa_{\boldsymbol\zeta}:I\to\Qlbar^\times$
be the tame character with
$\kappa_{\boldsymbol\zeta}(\sigma_{\mathrm{top}})
=\Lambda_{\boldsymbol\zeta}^{-v}$. For $\sigma\in I$ set
\[
 C_{\boldsymbol\zeta}(\sigma)
 :=\frac1{K_0}\prod_{i=0}^n(\zeta_i-1)\cdot
 \begin{cases}
 \displaystyle\frac{\kappa_{\boldsymbol\zeta}(\sigma)-1}
                   {1-\Lambda_{\boldsymbol\zeta}},
 &\Lambda_{\boldsymbol\zeta}\ne1,\\[6pt]
 v\nu_\ell(\sigma),&\Lambda_{\boldsymbol\zeta}=1.
 \end{cases}
\]
Then
\[
 \Var(\sigma)(e_{\boldsymbol\zeta}^*)
 =C_{\boldsymbol\zeta}(\sigma)e_{\boldsymbol\zeta^{-1}}.
\]
\end{proposition}
\begin{proof}
By \cite[Exp.~XIII, (1.4.3.2)--(1.4.3.4)]{SGA7_II},
we have
\[
\Var(\sigma\sigma')=\Var(\sigma)+\sigma\Var(\sigma').
\]

Since wild inertia acts trivially on $V$ and $W$,
the restriction of $\Var$ to wild inertia is a
continuous homomorphism into the additive group
$\operatorname{Hom}_{\Qlbar}(W,V)$.
As $\ell\ne p$, this homomorphism is zero.
Thus $\Var$ factors through tame inertia.

The stated formula satisfies the composition identity
and agrees with the formula for $\sigma_{\mathrm{top}}$.
The result follows by continuity.
\end{proof}

\subsection{Frobenius action}

Assume that the residue field is $\mathbb F_q$, and retain the
local diagonal presentation \eqref{eq:local_equation} over $R$.
Put $m=\operatorname{lcm}(k_0,\ldots,k_n)$ and $w_i=m/k_i$.
After adjoining $\varpi$ with $\varpi^m=\pi$, the substitution
$x_i=\varpi^{vw_i}y_i$ identifies the geometric generic fibre
with that of
\begin{equation}\label{eq:good_reduction_eq}
Y:\quad \sum_{i=0}^n a_iy_i^{k_i}+u=0.
\end{equation}
This hypersurface has good reduction, and we also write $a_i,u$
for their reductions in $\mathbb F_q^\times$.
The geometric deck transformations act on its special fibre by
$y_i\mapsto\zeta^{-vw_i}y_i$ for
$\zeta\in\mu_m(\overline{\mathbb F}_q)$.
Since this is an action of the group scheme $\mu_m$ over $\mathbb F_q$,
the local comparison of Subsection~\ref{subsec:local_diagonal} gives
an isomorphism of unramified Galois representations
\[
H^n_{\{0\}}(X_{\loc,\overline s},R\Psi\Qlbar)^I
\simeq H_c^n(Y_{\overline{\mathbb F}_q},\Qlbar)^{\mu_m}.
\]

We compute the action of Frobenius on this invariant part
by decomposing $H_c^n(Y_{\overline{\mathbb F}_q},\Qlbar)$
under the full diagonal group
$G_{\mathrm{diag}}=\prod_{i=0}^n\mu_{k_i}$, which acts on $Y$
by independent coordinate scalings.

The following trace formula for diagonal hypersurfaces is
classical, essentially due to Weil~\cite{Wei49}.
We give an alternative proof using exponential cohomology,
keeping track of the individual $G_{\mathrm{diag}}$-characters.

\begin{theorem}[Diagonal hypersurface trace formula]
\label{thm:trace_formula}
Let $n\geq1$, and let $Y\subset\mathbb A^{n+1}_{\mathbb F_q}$ be
the smooth affine hypersurface
\[
\sum_{i=0}^n a_ix_i^{k_i}+u=0,
\qquad a_i,u\in\mathbb F_q^\times,\quad p\nmid k_i.
\]
For $r\geq1$, put $Q=q^r$ and choose a nontrivial additive
character $\psi:\mathbb F_Q\to\Qlbar^\times$. We use the Gauss-sum convention
\begin{equation}\label{eq:def_Gauss_sum}
G(\chi,\psi):=-\sum_{z\in\mathbb F_Q^\times}\chi(z)\psi(z),
\qquad G(1,\psi)=1.
\end{equation}
Then, writing $\chi_{\mathrm{tot}}=\prod_{i=0}^n\chi_i$,
\[
\begin{aligned}
\operatorname{Tr}\!\left(\Frob_q^r\mid
H_c^n(Y_{\overline{\mathbb F}_q},\Qlbar)\right)
=\frac1Q\!
\sum_{\substack{\chi_i\in\widehat{\mathbb F_Q^\times}\\
                 \chi_i^{k_i}=1,\ \chi_i\ne1}}
\chi_{\mathrm{tot}}(u)G(\chi_{\mathrm{tot}}^{-1},\psi)
\prod_{i=0}^n\chi_i(a_i^{-1})G(\chi_i,\psi).
\end{aligned}
\]
\end{theorem}

\begin{proof}
Work over $\mathbb F_Q$, so that $\Frob_Q=\Frob_q^r$.
Put $m=\operatorname{lcm}(k_0,\ldots,k_n)$ and
\[
\mathcal U=\mathbb G_{m,t}\times\mathbb A_x^{n+1},
\qquad F(t,x)=ut^m+\sum_{i=0}^n a_ix_i^{k_i}.
\]
We compute $H_c^{n+2}(\mathcal U,F^*\mathcal L_\psi)^{\mu_m}$
in two ways.

For a character $\eta$ of $\mu_d$, let $\mathcal K_\eta$
be the corresponding Kummer summand of $[d]_*\Qlbar$, and put
\[
V_\eta(a)=H_c^1\bigl(\mathbb G_m,
\mathcal K_\eta\otimes\mathcal L_\psi(az)\bigr).
\]
We use the labeling for which, when $\eta$ is
$\Frob_Q$-stable, the geometric Frobenius trace of
$\mathcal K_\eta$ at $b\in\mathbb F_Q^\times$ is
\[
\chi_{\eta,Q}(b)=\eta(t^{Q-1}),\qquad t^d=b.
\]
This is independent of $t$, since $\eta^{Q-1}=1$.
This compactly supported cohomology is concentrated in degree one, and $V_\eta(a)$ is one-dimensional; for a power map on $\mathbb A^1$, only
nontrivial Kummer characters occur in degree one
(cf.~\cite[Sommes trig., \S4]{SGA4.5}).
Kummer decomposition and the K\"unneth formula give
\begin{equation}\label{eq:trace_W_Kummer_decomposition}
H_c^{n+2}(\mathcal U,F^*\mathcal L_\psi)
\simeq
\bigoplus_{\substack{\eta\in\widehat{\mu_m}\\
\eta_i\in\widehat{\mu_{k_i}}\setminus\{1\}}}
V_\eta(u)\otimes\bigotimes_{i=0}^nV_{\eta_i}(a_i).
\end{equation}
The action
$\zeta\cdot(t,x_i)=(\zeta t,\zeta^{m/k_i}x_i)$ of $\mu_m$
on the summand indexed by $(\eta,\eta_i)$ is through
$\eta(\zeta)\prod_i\eta_i(\zeta^{m/k_i})$.
Setting
\[
\eta_{\mathrm{tot}}(\zeta)
=\prod_{i=0}^n\eta_i(\zeta^{m/k_i}),
\]
we obtain
\begin{equation}\label{eq:trace_W_invariants_decomposition}
H_c^{n+2}(\mathcal U,F^*\mathcal L_\psi)^{\mu_m}
\simeq
\bigoplus_{\eta_i\ne1}
V_{\eta_{\mathrm{tot}}^{-1}}(u)
\otimes\bigotimes_{i=0}^nV_{\eta_i}(a_i).
\end{equation}
The summands are one-dimensional $G_{\mathrm{diag}}$-character spaces. The operator $\Frob_Q$ permutes them by
$\eta_i\mapsto\eta_i^Q$; hence only the fixed tuples
contribute to its trace. They correspond, through the characters $\chi_{\eta,Q}$ defined
above, to characters
$\chi_i\in\widehat{\mathbb F_Q^\times}$ with
$\chi_i^{k_i}=1$ and $\chi_i\ne1$.
On a $\Frob_Q$-stable eigenspace,
\[
\operatorname{Tr}(\Frob_Q\mid V_\eta(a))
=\chi_{\eta,Q}(a^{-1})G(\chi_{\eta,Q},\psi).
\]
Consequently, the contribution of such a tuple is
\[
\chi_{\mathrm{tot}}(u)G(\chi_{\mathrm{tot}}^{-1},\psi)
\prod_{i=0}^n\chi_i(a_i^{-1})G(\chi_i,\psi),
\qquad\chi_{\mathrm{tot}}=\prod_i\chi_i.
\]

On the other hand, set $x_i=t^{m/k_i}y_i$ and $s=t^m$,
so that $F=t^mh(y)$ with $h(y)=u+\sum_i a_iy_i^{k_i}$.
Taking $\mu_m$-invariants and integrating in the
$s$-variable yield
\begin{equation}\label{eq:trace_W_Y_comparison}
\begin{aligned}
H_c^{n+2}(\mathcal U,F^*\mathcal L_\psi)^{\mu_m}
&\simeq H_c^{n+2}\bigl(
\mathbb G_{m,s}\times\mathbb A_y^{n+1},
\mathcal L_\psi(sh(y))\bigr)\\
&\simeq H_c^n(Y,\Qlbar)(-1).
\end{aligned}
\end{equation}
Indeed, integration over $\mathbb A_s^1$ gives
$Rp_!\mathcal L_\psi(sh)=i_*\Qlbar(-1)[-2]$,
where $i:Y=\{h=0\}\hookrightarrow\mathbb A_y^{n+1}$.
Removing $s=0$ does not affect degree $n+2$, since $n\ge1$.

These comparisons are $G_{\mathrm{diag}}$- and
$\Frob_Q$-equivariant. For
$\boldsymbol\eta=(\eta_0,\ldots,\eta_n)$, with every
$\eta_i\ne1$, let $L_{\boldsymbol\eta}$ be the corresponding
$G_{\mathrm{diag}}$-character space in
$H_c^n(Y_{\overline{\mathbb F}_q},\Qlbar)$. It is one-dimensional,
and \eqref{eq:trace_W_invariants_decomposition} identifies
\[
L_{\boldsymbol\eta}(-1)
\simeq V_{\eta_{\mathrm{tot}}^{-1}}(u)
\otimes\bigotimes_{i=0}^nV_{\eta_i}(a_i).
\]
For a tuple fixed by $\Frob_Q$, writing
$\chi_i=\chi_{\eta_i,Q}$ and
$\chi_{\mathrm{tot}}=\prod_i\chi_i$, we obtain
\begin{equation}\label{eq:individual_diagonal_Frob_trace}
\operatorname{Tr}(\Frob_Q\mid L_{\boldsymbol\eta})
=\frac{1}{Q}\chi_{\mathrm{tot}}(u)
G(\chi_{\mathrm{tot}}^{-1},\psi)
\prod_{i=0}^n\chi_i(a_i^{-1})G(\chi_i,\psi).
\end{equation}
Summing over the Frobenius-stable tuples proves the theorem.
\end{proof}

The group $G_{\mathrm{diag}}$ acts on both
$X_{\loc}$ and $Y$ by coordinate scalings, and the
local comparison and the change of variables
$x_i=\varpi^{vw_i}y_i$ are equivariant for these actions. Choose a primitive $m$-th root of unity $\omega$
compatible with the cyclic generators of
\S\ref{subsec:diag_singularity_cl}.
For $\zeta_i\in\mu_{k_i}\setminus\{1\}$, let
$\eta_{\zeta_i}$ be the character of $\mu_{k_i}$
determined by
\[
\eta_{\zeta_i}(\omega^{w_i})=\zeta_i^{-1},
\]
and write $L_{\boldsymbol\zeta}
=L_{(\eta_{\zeta_0},\ldots,\eta_{\zeta_n})}$.
Under the $G_{\mathrm{diag}}$-equivariant comparison
and Poincar\'e--Lefschetz duality, these character
spaces correspond to the local eigenlines
$\Qlbar e_{\boldsymbol\zeta}$.
The deck transformation
$y_i\mapsto\omega^{-vw_i}y_i$ acts on
$L_{\boldsymbol\zeta}$ by
$\Lambda_{\boldsymbol\zeta}^{v}$, as does
$\sigma_{\mathrm{top}}$ on $e_{\boldsymbol\zeta}$.
Moreover, geometric Frobenius permutes these lines by
\[
\Frob_q(L_{\boldsymbol\zeta})
=L_{\boldsymbol\zeta^q}.
\]

\begin{proposition}[Inertia invariants]
\label{prop:inertia_invariants}
With the notation of Theorem~\ref{thm:local_diagonal_main},
put $\Lambda_{\boldsymbol\zeta}=\prod_{i=0}^n\zeta_i$. Then
\[
H^n_{\{0\}}(X_{\loc,\overline s},R\Psi(\Qlbar))^I
=\bigoplus_{\substack{\boldsymbol\zeta\in\Xi\\
\Lambda_{\boldsymbol\zeta}^v=1}}
\Qlbar\,e_{\boldsymbol\zeta}.
\]
\end{proposition}
\begin{proof}
The tame generator acts on $e_{\boldsymbol\zeta}$ by
$\Lambda_{\boldsymbol\zeta}^v$ by
Theorem~\ref{thm:local_diagonal_main}.
\end{proof}

For $r\ge1$, put $Q=q^r$ and fix a nontrivial additive
character $\psi:\mathbb F_Q\to\Qlbar^\times$.
For nontrivial characters $\chi_0,\ldots,\chi_n$ of
$\mathbb F_Q^\times$, define the Jacobi sum by
\begin{equation}\label{eq:Jacobi_sum_def}
J_r(\chi_0,\ldots,\chi_n)
:=(-1)^n\!\sum_{\substack{t_i\in\mathbb F_Q^\times\\
t_0+\cdots+t_n=-1}}\ \prod_{i=0}^n\chi_i(t_i).
\end{equation}
Writing $\chi_{\mathrm{tot}}=\prod_{i=0}^n\chi_i$
and using the Gauss-sum convention
\eqref{eq:def_Gauss_sum}, character orthogonality gives
\begin{equation}\label{eq:Jacobi_Gauss_relation}
J_r(\chi_0,\ldots,\chi_n)
=\frac1Q\,G(\chi_{\mathrm{tot}}^{-1},\psi)
\prod_{i=0}^nG(\chi_i,\psi).
\end{equation}

Put
\[
\Xi_{q^r}^{\mathrm{inv}}
=\bigl\{\boldsymbol\zeta\in\Xi:\
\zeta_i^{q^r-1}=1\ \text{for all }i,\quad
\Lambda_{\boldsymbol\zeta}^v=1\bigr\}.
\]
For $\boldsymbol\zeta\in\Xi_{q^r}^{\mathrm{inv}}$, put
$\chi_{\zeta_i,r}=\chi_{\eta_{\zeta_i},q^r}$ and
$\chi_{\boldsymbol\zeta,r}=\prod_i\chi_{\zeta_i,r}$.

\begin{theorem}[Frobenius trace on inertia invariants]
\label{thm:Frob_trace_on_inertia_inv}
For every $r\ge1$,
\[
\begin{aligned}
&\operatorname{Tr}\bigl(\Frob_q^r\mid
H^n_{\{0\}}(X_{\loc,\overline s},R\Psi(\Qlbar))^I\bigr)\\
&\qquad=\sum_{\boldsymbol\zeta\in\Xi_{q^r}^{\mathrm{inv}}}
\chi_{\boldsymbol\zeta,r}(u)
\left(\prod_{i=0}^n\chi_{\zeta_i,r}(a_i^{-1})\right)
J_r(\chi_{\zeta_0,r},\ldots,\chi_{\zeta_n,r}).
\end{aligned}
\]
\end{theorem}

\begin{proof}
By Proposition~\ref{prop:inertia_invariants} and the
$G_{\mathrm{diag}}$-equivariant comparison above,
the inertia-invariant part is a sum of the character
lines $L_{\boldsymbol\zeta}$ with
$\Lambda_{\boldsymbol\zeta}^v=1$.
Frobenius permutes these lines, so only those indexed
by $\Xi_{q^r}^{\mathrm{inv}}$ contribute to its trace.
The formula follows from
\eqref{eq:individual_diagonal_Frob_trace}
and \eqref{eq:Jacobi_Gauss_relation}.
\end{proof}

\begin{remark}[Zeta factor of the inertia invariants]
Put $V=H_c^n(X_{\loc,\overline\eta},\Qlbar)^I$ and
$F=\Frob_q|_V$. Its cohomological zeta factor is
\[
Z(V,t):=\exp\!\left(\sum_{r\ge1}
\operatorname{Tr}(F^r\mid V)\frac{t^r}{r}\right)
=\det(1-tF\mid V)^{-1}.
\]
For an orbit $\mathcal O$ of
$\boldsymbol\zeta\mapsto\boldsymbol\zeta^q$ on
$\Xi^{\mathrm{inv}}:=\{\boldsymbol\zeta\in\Xi:
\Lambda_{\boldsymbol\zeta}^v=1\}$,
let $d_{\mathcal O}=|\mathcal O|$ and let
$\lambda_{\mathcal O}$ be the eigenvalue of
$F^{d_{\mathcal O}}$ on any line of that orbit. Then
\[
Z(V,t)=\prod_{\mathcal O\in
\Xi^{\mathrm{inv}}/\langle q\rangle}
(1-\lambda_{\mathcal O}t^{d_{\mathcal O}})^{-1},
\]
where $\lambda_{\mathcal O}$ is given by the characterwise
Gauss--Jacobi calculation in the proof of
Theorem~\ref{thm:trace_formula}.
\end{remark}

\subsection{Main theorem}

The local diagonal calculations give explicit tame monodromy,
intersection pairings, and variation maps. Together with the
specialization sequence and local duality, they determine the
inertia action on the middle cohomology of the smooth generic fibre.
When the residue field is finite, the diagonal-character decomposition
also yields a Frobenius trace formula for the inertia-invariant local
cohomology. We collect these results in the following theorem.

\begin{theorem}[Generalized arithmetic Picard--Lefschetz formula]
\label{thm:main_gen_PL}
Let $f:X\to S=\Spec R$ be proper and flat of relative dimension
$n\geq1$, with smooth geometric generic fibre, where $R$ is a
henselian discrete valuation ring whose fraction field has characteristic zero,
uniformizer $\pi$, and residue characteristic $p$. Fix $\ell\ne p$.
Assume that $f$ is smooth away from a finite set
$\Sigma\subset X_{\overline s}$ of geometric points.
Suppose that, at each singular point, the family admits
an \'etale-local presentation over $R$ of the form
\[
\sum_{i=0}^n a_{i,x}z_i^{k_{i,x}}+u_x\pi^{v_x}=0,
\]
where $a_{i,x},u_x\in R^\times$, the integers
$k_{i,x}\ge2$ are invertible in $R$, and $v_x\ge1$.
Set
\[
 \begin{gathered}
 V_x=H^n_{\{x\}}(X_{\overline s},R\Psi\Qlbar),\qquad
 W_x=(R^n\Phi\Qlbar)_x,\\
 \Xi_x=\prod_{i=0}^n(\mu_{k_{i,x}}\setminus\{1\}),\qquad
 K_x=\prod_{i=0}^n k_{i,x},\qquad
 \Lambda_{\boldsymbol\zeta,x}=\prod_{i=0}^n\zeta_i.
 \end{gathered}
\]
We use the tame generator $\sigma_{\mathrm{top}}$, the map $\nu_\ell$,
and the scalar coordinate $[\,\cdot\,]_\tau$ fixed above.

\begin{enumerate}
\item \emph{Local monodromy and intersection form.}
Each $V_x$ has a basis
$\{e_{\boldsymbol\zeta,x}\}_{\boldsymbol\zeta\in\Xi_x}$ satisfying
\[
\sigma_{\mathrm{top}}e_{\boldsymbol\zeta,x}
=\Lambda_{\boldsymbol\zeta,x}^{v_x}e_{\boldsymbol\zeta,x},
\]
and
\[
 (e_{\boldsymbol\zeta,x},e_{\boldsymbol\xi,x})_x
 =\begin{cases}
 \displaystyle (-1)^{n(n+1)/2}K_x
 \frac{\Lambda_{\boldsymbol\zeta,x}-1}
 {\prod_{i=0}^n(\zeta_i-1)},
 &\boldsymbol\xi=\boldsymbol\zeta^{-1},\\[6pt]
 0,&\text{otherwise}.
 \end{cases}
\]
Here $(\ ,\ )_x$ is defined in
\eqref{eq:local_intersection_definition}. Write
$\delta_{\boldsymbol\zeta,x}=q_x(e_{\boldsymbol\zeta,x})$
for the image under the map $q_x$ forgetting supports;
these global classes need not be independent.

\item \emph{Specialization and duality.}
There is an exact sequence
\[
\begin{aligned}
0\longrightarrow H^n(X_{\overline s},\Qlbar)
 &\xrightarrow{sp}H^n(X_{\overline\eta},\Qlbar)
 \xrightarrow{can}\bigoplus_{x\in\Sigma}W_x\\
 &\xrightarrow{\partial}H^{n+1}(X_{\overline s},\Qlbar)
 \xrightarrow{sp}H^{n+1}(X_{\overline\eta},\Qlbar)
 \longrightarrow0,
\end{aligned}
\]
and $sp$ is an isomorphism in all other degrees.
Under $W_x\simeq V_x^\vee(-n)$, let
$\{e^*_{\boldsymbol\zeta,x}\}$ be the dual basis normalized by
$\langle e^*_{\boldsymbol\zeta,x},e_{\boldsymbol\xi,x}\rangle_{\mathrm{loc},x}
 =\delta_{\boldsymbol\zeta,\boldsymbol\xi}\tau$.
For $\alpha\in H^n(X_{\overline\eta},\Qlbar)$ put
\[
 b_{\boldsymbol\zeta,x}(\alpha)
 :=[\langle\alpha,\delta_{\boldsymbol\zeta,x}\rangle_X]_\tau.
\]
Then 
\[
can_x(\alpha)=\sum_{\boldsymbol\zeta\in\Xi_x}
 b_{\boldsymbol\zeta,x}(\alpha)e^*_{\boldsymbol\zeta,x}.
\]

\item \emph{Arithmetic variation.}
The inertia action on $H^n(X_{\overline\eta},\Qlbar)$ is tame.
Let $\kappa_{\boldsymbol\zeta,x}$ be the tame inertia character satisfying
$\kappa_{\boldsymbol\zeta,x}(\sigma_{\mathrm{top}})
 =\Lambda_{\boldsymbol\zeta,x}^{-v_x}$. Define
\[
 C_{\boldsymbol\zeta,x}(\sigma)
 =\frac{\prod_{i=0}^n(\zeta_i-1)}{K_x}
 \begin{cases}
 \displaystyle\frac{\kappa_{\boldsymbol\zeta,x}(\sigma)-1}
 {1-\Lambda_{\boldsymbol\zeta,x}},
 &\Lambda_{\boldsymbol\zeta,x}\ne1,\\[6pt]
 v_x\nu_\ell(\sigma),&\Lambda_{\boldsymbol\zeta,x}=1.
 \end{cases}
\]
Then, for $\sigma\in I$ and $\alpha\in H^n(X_{\overline\eta},\Qlbar)$,
\[
 \sigma(\alpha)=\alpha+
 \sum_{x\in\Sigma}\sum_{\boldsymbol\zeta\in\Xi_x}
 C_{\boldsymbol\zeta,x}(\sigma)
 b_{\boldsymbol\zeta,x}(\alpha)
 \delta_{\boldsymbol\zeta^{-1},x}.
\]

\item \emph{Frobenius trace.}
Suppose that the residue field is $\mathbb F_q$,
and let $\Frob_q$ denote geometric Frobenius.
For each $x\in\Sigma\cap X_{s}(\mathbb{F}_{q})$, we have
\[
V_x^I=
\bigoplus_{\substack{\boldsymbol\zeta\in\Xi_x\\
\Lambda_{\boldsymbol\zeta,x}^{v_x}=1}}
\Qlbar\,e_{\boldsymbol\zeta,x}.
\]
This space is stable under $\Frob_q$.
For $r\geq1$, write $Q=q^r$ and set
\[
\Xi_{x,Q}^{\mathrm{inv}}
=\bigl\{\boldsymbol\zeta\in\Xi_x:\
\zeta_i^{Q-1}=1\ \text{for all }i,\quad
\Lambda_{\boldsymbol\zeta,x}^{v_x}=1\bigr\}.
\]
Put $m_x=\operatorname{lcm}_i(k_{i,x})$ and
$w_{i,x}=m_x/k_{i,x}$, and choose a primitive
$m_x$-th root of unity $\omega_x$ compatible with
the local eigenbasis.
For $\boldsymbol\zeta\in\Xi_{x,Q}^{\mathrm{inv}}$, let
$\eta_{\zeta_i,x}$ be the character of $\mu_{k_{i,x}}$
determined by
$\eta_{\zeta_i,x}(\omega_x^{w_{i,x}})=\zeta_i^{-1}$, and define
the multiplicative character of $\mathbb F_Q^\times$
\[
\chi_{\zeta_i,x,r}(b)
=\eta_{\zeta_i,x}(t^{Q-1}),\qquad t^{k_{i,x}}=b.
\]
This is independent of $t$ since $\zeta_i^{Q-1}=1$.
Write $\chi_{\boldsymbol\zeta,x,r}
=\prod_{i=0}^n\chi_{\zeta_i,x,r}$, and, for nontrivial
multiplicative characters $\chi_0,\ldots,\chi_n$ of
$\mathbb F_Q^\times$, set
\[
J_r(\chi_0,\ldots,\chi_n)
=(-1)^n\!\sum_{\substack{t_i\in\mathbb F_Q^\times\\
t_0+\cdots+t_n=-1}}\ \prod_{i=0}^n\chi_i(t_i).
\]
Then
\[
\begin{aligned}
\operatorname{Tr}(\Frob_q^r\mid V_x^I)
=\sum_{\boldsymbol\zeta\in\Xi_{x,Q}^{\mathrm{inv}}}
\chi_{\boldsymbol\zeta,x,r}(\overline u_x)
\left(\prod_{i=0}^n
\chi_{\zeta_i,x,r}(\overline a_{i,x}^{-1})\right)
J_r(\chi_{\zeta_0,x,r},\ldots,\chi_{\zeta_n,x,r}).
\end{aligned}
\]
\end{enumerate}
\end{theorem}
\begin{proof}
Part (1) follows from
Theorem~\ref{thm:local_diagonal_main}.
The specialization sequence in (2) is
Proposition~\ref{prop:specialization_exact_sequence},
and the formula for $can_x$ follows from local duality
and the adjointness of $can_x$ and $q_x$.
Part (3) follows from
Proposition~\ref{prop:local_variation_formula}
and the decomposition \eqref{diag:var_decomposition}.
Finally, part (4) is a direct consequence of
Theorem~\ref{thm:Frob_trace_on_inertia_inv}.
\end{proof}

\section{Applications to Airy moments}\label{sec:application}
We now apply the generalized arithmetic Picard--Lefschetz formula to
Galois representations arising from Airy moments. These moments are
related to the cohomology of certain affine
hypersurfaces. The natural finite-group actions on these varieties
allow us to isolate the cohomological pieces associated with Airy
moments.

We begin by recalling the Airy sheaf and a modified Airy sheaf related
to it by a cubic covering, together with their monodromy and the
cohomology of their symmetric powers. We then describe the geometric
realizations of Airy moments and their middle cohomology. Finally, we
apply the arithmetic Picard--Lefschetz formula of
Section~\ref{sec:diag_singularity} to the singularities arising in
reduction modulo $p$, and study the resulting inertia and Frobenius
actions.

\subsection{Airy sheaves}
\label{subsec:airy_sheaves}
Let $\mathbb F_q$ be a finite field of characteristic $p\ne3$, let
$\ell\ne p$, and fix a nontrivial additive character
$\psi:\mathbb F_q\to\Qlbar^\times$. For a function $h$, write
$\mathcal L_{\psi(h)}$ for the associated Artin--Schreier sheaf.
The rank-two Airy sheaf is obtained by applying the
Fourier--Deligne transform to
$\mathcal L_{\psi(x^3/3)}$.
Explicitly, on $\mathbb A^1_z$ it is given by
\[
  \Ai=R^1(\pi_z)_!\mathcal L_{\psi(x^3/3-zx)},
  \qquad
  \pi_z:\mathbb A^1_x\times\mathbb A^1_z\longrightarrow\mathbb A^1_z.
\]
Its de Rham counterpart, the Airy connection, is considered in
\cite[\S6A]{SY23}; we work here with the $\ell$-adic sheaf.
We also use the modified Airy sheaf on $\Gm_{z'}$:
\[
  \Ai'=R^1\pi_!\mathcal L_{\psi(x^3/(3z')-x)},
  \qquad
  \pi:\mathbb A^1_x\times\Gm_{z'}\longrightarrow\Gm_{z'}.
\]

Both $\Ai$ and $\Ai'$ are symplectically self-dual,
with
\[
\det\Ai\simeq\det\Ai'\simeq\Qlbar(-1).
\]
Indeed, the involution $x\mapsto-x$ identifies the
Artin--Schreier sheaf defining each of them with its dual,
and Poincar\'e duality induces a perfect alternating
pairing with values in $\Qlbar(-1)$.

\begin{proposition}\label{prop:airy_cubic_pullback}
Let $[3]:\Gm_z\to\Gm_{z'}$ be the map $z'=z^3$. Then
\[
  [3]^*\Ai'\simeq\Ai|_{\Gm_z}.
\]
\end{proposition}
\begin{proof}
By base change for compactly supported direct images and the change of
variables $x=zy$, the phase $x^3/(3z^3)-x$ becomes $y^3/3-zy$.
This gives the asserted isomorphism.
\end{proof}

\begin{proposition}[Cubic-cover decomposition]
\label{prop:airy_cubic_decomposition}
For every $k\ge1$, there is a Frobenius-equivariant isomorphism
\[
  R\Gamma_c\bigl(\Gm_{\overline{\mathbb F}_q},\Sym^k\Ai\bigr)
  \simeq
  R\Gamma_c\bigl(\Gm_{\overline{\mathbb F}_q},
     \Sym^k\Ai'\otimes[3]_*\Qlbar\bigr).
\]
If $q\equiv1\pmod3$, fix a character
$\chi_3:\mathbb F_q^\times\to\Qlbar^\times$ of order $3$. Then
\begin{equation}\label{eq:airy_cubic_decomposition}
  H_c^1\bigl(\Gm_{\overline{\mathbb F}_q},\Sym^k\Ai\bigr)
  \simeq
  \bigoplus_{j=0}^2
  H_c^1\bigl(\Gm_{\overline{\mathbb F}_q},
     \mathcal K_{\chi_3^j}\otimes\Sym^k\Ai'\bigr),
\end{equation}
with its Frobenius action. Over $\overline{\mathbb F}_q$ the same
three-character decomposition exists for any $q$; if
$q\equiv2\pmod3$, Frobenius exchanges the two nontrivial summands.
\end{proposition}
\begin{proof}
The preceding proposition and the projection formula for the finite
\'{e}tale map $[3]$ give
\[
  R\Gamma_c(\Gm,[3]^*\Sym^k\Ai')
  \simeq
  R\Gamma_c(\Gm,\Sym^k\Ai'\otimes[3]_*\Qlbar).
\]
If $q\equiv1\pmod3$, the Kummer decomposition
$[3]_*\Qlbar\simeq\bigoplus_{j=0}^2\mathcal K_{\chi_3^j}$
is defined over $\mathbb F_q$, proving \eqref{eq:airy_cubic_decomposition}.
In general this decomposition is defined geometrically, and
Frobenius acts on its characters by the $q$-power map.
\end{proof}

\begin{proposition}\label{prop:Ai'_conf_hypergeo_iso}
Assume $q\equiv1\pmod3$, and fix a character $\chi_3$
of order $3$ of $\mathbb F_q^\times$. In the normalization
of \cite[Theorem~4.1.1 and \S10.4]{Kat88}, there is an
isomorphism of $\ell$-adic sheaves on $\Gm_{z'}$:
\[
  \Ai'\simeq[z'\mapsto z'/9]^*
       \operatorname{Kl}(\psi;\chi_3,\chi_3^2).
\]
Thus $\Ai'$ is a confluent hypergeometric sheaf
of type $(2,0)$.
\end{proposition}

\begin{proof}
We compare the Frobenius traces over every finite extension
$\mathbb F_Q/\mathbb F_q$. Write
\[
\psi_Q=\psi\circ\operatorname{Tr}_{\mathbb F_Q/\mathbb F_q},
\qquad
\chi_{3,Q}=\chi_3\circ\operatorname{Nm}_{\mathbb F_Q/\mathbb F_q}.
\]
We use the Gauss-sum convention
\eqref{eq:def_Gauss_sum}.

The Frobenius trace of $\Ai'$ at $z'\in\mathbb F_Q^\times$ is
\[
A_Q(z')=-\sum_{x\in\mathbb F_Q}
\psi_Q\left(\frac{x^3}{3z'}-x\right).
\]
For a multiplicative character $\eta$ of
$\mathbb F_Q^\times$, consider its Mellin transform
\[
\mathcal M_Q(\eta)
=\sum_{z'\in\mathbb F_Q^\times}A_Q(z')\eta(z').
\]
When $\eta\ne1$, the term $x=0$ vanishes. Substituting
$y=x^3/(3z')$ for $x\ne0$, we obtain
\[
\begin{aligned}
\mathcal M_Q(\eta)
&=-\sum_{x\ne0}\psi_Q(-x)\eta(x^3/3)
       \sum_{y\ne0}\eta^{-1}(y)\psi_Q(y)\\
&=-\eta(-1/3)\,
       G(\eta^3,\psi_Q)G(\eta^{-1},\psi_Q).
\end{aligned}
\]
The Hasse--Davenport relation gives
\[
G(\eta^3,\psi_Q)
=\frac{\eta(27)}{Q}
G(\eta,\psi_Q)
G(\eta\chi_{3,Q},\psi_Q)
G(\eta\chi_{3,Q}^2,\psi_Q).
\]
Together with
\[
G(\eta,\psi_Q)G(\eta^{-1},\psi_Q)
=\eta(-1)Q,
\qquad \eta\ne1,
\]
this yields
\begin{equation}\label{eq:Ai_prime_Mellin}
\mathcal M_Q(\eta)
=-\eta(9)\,
G(\eta\chi_{3,Q},\psi_Q)
G(\eta\chi_{3,Q}^2,\psi_Q).
\end{equation}
For $\eta=1$, the same identity follows from
$\mathcal M_Q(1)=-Q$ and
\[
G(\chi_{3,Q},\psi_Q)G(\chi_{3,Q}^2,\psi_Q)=Q.
\]

On the other hand, the Frobenius trace of
$\operatorname{Kl}(\psi;\chi_3,\chi_3^2)$ at
$w\in\mathbb F_Q^\times$ is
\[
K_Q(w)=-\sum_{ab=w}
\chi_{3,Q}(a)\chi_{3,Q}^2(b)\psi_Q(a+b).
\]
Its Mellin transform is
\[
\begin{aligned}
\sum_{w\ne0}K_Q(w)\eta(w)
&=-\left(\sum_{a\ne0}
  (\eta\chi_{3,Q})(a)\psi_Q(a)\right)\cdot
  \left(\sum_{b\ne0}
  (\eta\chi_{3,Q}^2)(b)\psi_Q(b)\right)\\
&=-G(\eta\chi_{3,Q},\psi_Q)
    G(\eta\chi_{3,Q}^2,\psi_Q).
\end{aligned}
\]
Replacing $w$ by $z'/9$ introduces the factor $\eta(9)$.
Thus \eqref{eq:Ai_prime_Mellin} shows that $A_Q(z')$
and $K_Q(z'/9)$ have the same Mellin transforms.
By character orthogonality, their Frobenius traces agree
at every point over every finite extension of $\mathbb F_q$.

By Chebotarev, the two lisse sheaves have isomorphic
semisimplifications. Since the rank-two Kloosterman sheaf
is geometrically irreducible
\cite[Theorem~8.4.2]{Kat90},
they are isomorphic.
\end{proof}

\begin{proposition}\label{prop:G_geom_Ai_Ai'}
Let $G_{\geom}$ and $G_{\geom}'$ be the geometric monodromy groups
of $\Ai$ and $\Ai'$, respectively. Then
\[
\begin{array}{c|cc}
 &G_{\geom}&G_{\geom}'\\ \hline
p>5&\SL(2,\Qlbar)&\SL(2,\Qlbar)\\
p=5&\SL(2,\mathbb F_5)&\SL(2,\mathbb F_5)\\
p=2&Q_8&\SL(2,\mathbb F_3).
\end{array}
\]
Here the finite groups act through their faithful two-dimensional
symplectic representations over $\Qlbar$.
\end{proposition}

\begin{proof}
After a finite extension of the base field,
we may assume that $q\equiv1\pmod3$.
The groups for $\Ai$ are determined by Katz--Tiep
\cite[Theorems~10.2.4(iii) and 10.2.7(i),(iv)]{KT25}.
The identification of $\Ai'$ with a rank-two confluent
hypergeometric sheaf
(Proposition~\ref{prop:Ai'_conf_hypergeo_iso})
and \cite[Remark~10.2.5]{KT25} give
$G_{\geom}'=\SL(2,\Qlbar)$ when $p>5$.

For $p=2,5$, the cubic Kummer pullback
$[3]^*\Ai'\simeq\Ai|_{\Gm}$ implies
$G_{\geom}\triangleleft G_{\geom}'$ with cyclic quotient
of order dividing three. Both groups are subgroups of
$\SL(2,\Qlbar)$. By the classification of finite subgroups
of $\SL(2,\Qlbar)$
\cite[Theorem~6.11]{MR2919145},
the binary icosahedral group $\SL(2,\mathbb F_5)$ is maximal
among finite subgroups, giving
$G_{\geom}'=G_{\geom}$ when $p=5$.

For $p=2$, the possibilities are $Q_8$ and its
index-three extension
$Q_8\rtimes C_3\simeq\SL(2,\mathbb F_3)$.
Their faithful two-dimensional representations $V$ satisfy
\[
\dim(\Sym^4 V)^{Q_8}=2,\qquad
\dim(\Sym^4 V)^{\SL(2,\mathbb F_3)}=0.
\]
For every $q=2^r$, Lemma~\ref{lem:airy_prime_4th_moment}
gives $M_{\Ai'}^4(q)=-q^2$.
By the trace formula and weight considerations, $H_c^2(\Gm_{\overline{\mathbb F}_2},\Sym^4\Ai')=0$, excluding $Q_8$.
\end{proof}

\begin{lemma}\label{lem:airy_prime_4th_moment}
For every $q=2^r$, the fourth moment of $\Ai'$ is
\[
M_{\Ai'}^4(q)
:=\sum_{a\in\mathbb F_q^\times}
\operatorname{Tr}(\Frob_a\mid\Sym^4\Ai')=-q^2.
\]
\end{lemma}

\begin{proof}
Put
\[
S_a=\sum_{x\in\mathbb F_q}\psi(a^{-1}x^3+x).
\]
Since $\Ai'$ is symplectically self-dual with
$\det\Ai'\simeq\Qlbar(-1)$, its Frobenius
eigenvalues $\alpha_a,\beta_a$ satisfy
$\alpha_a+\beta_a=-S_a$ and $\alpha_a\beta_a=q$. Consequently,
\[
\operatorname{Tr}(\Frob_a\mid\Sym^4\Ai')
=S_a^4-3qS_a^2+q^2.
\]

Let $\varepsilon=\gcd(3,q-1)-1\in\{0,2\}$ and put
\[
W_j=\sum_{\substack{x_1,\ldots,x_j\in\mathbb F_q\\
                    \sum_i x_i^3=0}}
\psi\!\left(\sum_i x_i\right).
\]
Additive-character orthogonality gives
$\sum_{a\ne0}S_a^j=qW_j$ for $j=2,4$.
Since $x_1^3=x_2^3$ precisely when
$x_1=\omega x_2$ for some $\omega^3=1$, we obtain
$W_2=q-\varepsilon$.

To compute $W_4$, put
$U=x_1+x_2$ and $V=x_3+x_4$. In characteristic two,
\[
\sum_{i=1}^4x_i^3
=U^3+V^3+Ux_1x_2+Vx_3x_4.
\]
For $U\ne0$, the polynomial $X^2+UX+P$ has
$1+\chi(P/U^2)$ roots, where
$\chi(t)=(-1)^{\operatorname{Tr}_{\mathbb F_q/\mathbb F_2}(t)}$.
Let $N(U,V)$ be the number of quadruples with the given
values of $U,V$ and $\sum_i x_i^3=0$.

If $U=V=0$, then $N(U,V)=q^2$.
If exactly one of $U,V$ vanishes, then
$N(U,V)=q\varepsilon$.
For $UV\ne0$, writing $P=x_1x_2$ gives
\[
N(U,V)=\sum_{P\in\mathbb F_q}
\left(1+\chi(P/U^2)\right)
\left(1+\chi\left(
\frac{U^3+V^3+UP}{V^3}\right)\right)
=q\bigl(1+\mathbf 1_{U^3=V^3}\bigr).
\]
Therefore
\[
\begin{aligned}
W_4
&=\sum_{U,V\in\mathbb F_q}\psi(U+V)N(U,V)\\
&=q^2-2q\varepsilon
  +q\bigl[1+(q-1)-\varepsilon\bigr]=2q^2-3q\varepsilon,
\end{aligned}
\]
where we use $\sum_{t\ne0}\psi(t)=-1$ and
$\sum_{U,V\ne0,\ U^3=V^3}\psi(U+V)
=(q-1)-\varepsilon$.

It follows that
\[
\sum_{a\ne0}S_a^2=q^2-q\varepsilon,\qquad
\sum_{a\ne0}S_a^4=2q^3-3q^2\varepsilon.
\]
Substituting into the symmetric-power trace formula gives
$M_{\Ai'}^4(q)=-q^2$.
\end{proof}

For a lisse sheaf $\mathcal F$ on $U=\mathbb A^1$ or
$\Gm$, put
\[
h_c^i(U,\mathcal F)
=\dim H_c^i(U_{\overline{\mathbb F}_p},\mathcal F).
\]
By the Grothendieck--Ogg--Shafarevich formula and $H_c^0(U,\mathcal F)=0$,
\begin{equation}\label{eq:airy_GOS_dimensions}
h_c^1(U,\mathcal F)=
\begin{cases}
h_c^2(U,\mathcal F)
+\Swan_0(\mathcal F)+\Swan_\infty(\mathcal F),
&U=\Gm,\\[3pt]
h_c^2(U,\mathcal F)
+\Swan_\infty(\mathcal F)-\operatorname{rank}(\mathcal F),
&U=\mathbb A^1.
\end{cases}
\end{equation}
For $\mathcal F=\Sym^k\Ai'$ the ramification at $0$ is tame:
its pullback by the tame cubic cover extends lisse across $0$.
Thus $\Swan_0(\Sym^k\Ai')=0$. 
Since the geometric monodromy representations of $\Ai$
and $\Ai'$ are symplectic, their symmetric powers are
self-dual. By Poincar\'e duality,
\[
h_c^2(U,\Sym^k\mathcal A)
=\dim(\Sym^k V)^{G_{\geom}(\mathcal A)},
\qquad \mathcal A=\Ai,\Ai',
\]
where $V$ denotes the geometric generic stalk of $\mathcal A$.

\begin{proposition}\label{prop:dim_Sym_Ai'_F_p}
For $p>5$ and $k\ge1$,
\[
 h_c^1(\Gm,\Sym^k\Ai')
 =\frac12\left(k+1-\left\lfloor\frac{k}{p}\right\rfloor-\delta\right),
\quad
\delta=\begin{cases}
0,& k-\lfloor k/p\rfloor\text{ odd},\\
1,& k-\lfloor k/p\rfloor\text{ even}.
\end{cases}
\]
\end{proposition}
\begin{proof}
The standard representation of $G_{\geom}'=\SL(2,\Qlbar)$
has no invariant vectors in its positive symmetric powers, hence
$h_c^2=0$. The cubic pullback $[3]^*\Ai'\simeq\Ai|_{\Gm}$
multiplies the Swan conductor at infinity by three.
The formula of \cite{HRL11} therefore gives
\[
\Swan_\infty(\Sym^k\Ai')
=\frac13\Swan_\infty(\Sym^k\Ai)
=\frac12\left(k+1-\left\lfloor\frac{k}{p}\right\rfloor-\delta\right).
\]
The assertion follows from \eqref{eq:airy_GOS_dimensions}.
\end{proof}

\begin{proposition}\label{prop:dim_Aiprime_p5}
For $p=5$ and $k\ge0$, let $h_k=h_c^2(\Gm,\Sym^k\Ai')$.
Then
\[
\sum_{k\ge0}h_k t^k=\frac{1+t^{30}}{(1-t^{12})(1-t^{20})},
\]
and
\[
 h_c^1(\Gm,\Sym^k\Ai')
 =\frac12\left(k+1-\left\lfloor\frac{k}{5}\right\rfloor-\delta\right)+h_k,
\]
where $\delta=0$ if $k-\lfloor k/5\rfloor$ is odd, and
$\delta=1$ otherwise.
\end{proposition}
\begin{proof}
The Swan conductor is the same as in the preceding proof.
For $G=G_{\geom}'=\SL(2,\mathbb F_5)$, let $\rho$ denote the
faithful two-dimensional symplectic representation carried by $\Ai'$.
Its character values are recalled in Table~\ref{tab:char_table_SL2F5}.
\begin{table}[ht]
\centering
\renewcommand{\arraystretch}{1.2}
\begin{tabular}{c|ccccccccc}
\hline
Class & $I$ & $-I$ & $4A$ & $3A$ & $6A$ & $5A$ & $5B$ & $10A$ & $10B$ \\
Size & $1$ & $1$ & $30$ & $20$ & $20$ & $12$ & $12$ & $12$ & $12$ \\
\hline
$\rho$ & $2$ & $-2$ & $0$ & $-1$ & $1$ & $\alpha-1$ & $\beta-1$ & $1-\alpha$ & $1-\beta$ \\
$\rho'$ & $2$ & $-2$ & $0$ & $-1$ & $1$ & $\beta-1$ & $\alpha-1$ & $1-\beta$ & $1-\alpha$ \\\hline
\end{tabular}
\caption{Part of the character table of $\SL(2,\mathbb{F}_{5})$, where $\alpha=\frac{1+\sqrt{5}}{2}$ and $\beta=\frac{1-\sqrt{5}}{2}$.}
\label{tab:char_table_SL2F5}
\end{table}
By Molien's formula,
\[
 \sum_{k\ge0}h_k t^k
 =\frac1{|G|}\sum_{g\in G}\frac1{\det(1-t\rho(g))}
 =\frac1{120}\sum_{g\in G}\frac1{1-t\operatorname{Tr}\rho(g)+t^2}.
\]
Grouping the terms using the table gives
\[
\begin{aligned}
120\sum_{k\ge0}h_k t^k={}&
 \frac1{(1-t)^2}+\frac1{(1+t)^2}+\frac{30}{1+t^2}
 +\frac{20}{1+t+t^2}+\frac{20}{1-t+t^2}\\
 &+12\sum_{c\in\{\alpha-1,\beta-1,1-\alpha,1-\beta\}}
                  \frac1{1-ct+t^2}
 =\frac{120(1+t^{30})}{(1-t^{12})(1-t^{20})}.
\end{aligned}
\]
The dimension formula follows from \eqref{eq:airy_GOS_dimensions}.
\end{proof}

\begin{lemma}\label{lem:swan_symk_Q8}
In characteristic $2$, the Swan conductors of the symmetric powers
of $\Ai$ at infinity satisfy, for $m\ge0$,
\[
\Swan_\infty(\Sym^{2m+1}\Ai)=3(m+1),\qquad
\Swan_\infty(\Sym^{2m}\Ai)
     =3\left\lfloor\frac{m+1}{2}\right\rfloor.
\]
\end{lemma}
\begin{proof}
Let $V = 2_{0}$ be the faithful two-dimensional representation of
$G_{\geom}=Q_8$, and write
$W=1_1\oplus1_2\oplus1_3$ for the sum of its three
nontrivial one-dimensional representations, with notation as in Table~\ref{tab:char_table_Q8}.
\begin{table}[ht]
    \centering
    \renewcommand{\arraystretch}{1.2}
    \begin{tabular}{c|ccccc}
        \hline
        Class & $1A$ & $2A$ & $4A$ & $4B$ & $4C$ \\
        Size & $1$ & $1$ & $2$ & $2$ & $2$ \\
        \hline
        $1_{0}$ & $1$ & $1$ & $1$ & $1$ & $1$ \\
        $1_{1}$ & $1$ & $1$ & $1$ & $-1$ & $-1$ \\
        $1_{2}$ & $1$ & $1$ & $-1$ & $1$ & $-1$ \\
        $1_{3}$ & $1$ & $1$ & $-1$ & $-1$ & $1$ \\
        $2_{0}$ & $2$ & $-2$ & $0$ & $0$ & $0$ \\
        \hline
    \end{tabular}
    \caption{Character table of $Q_{8}$.}
    \label{tab:char_table_Q8}
\end{table}

Let $I_\infty$ and $P_\infty$ denote the inertia and wild
inertia groups at $z=\infty$, respectively.
Since $[3]^*\Ai'\simeq\Ai|_{\Gm}$ and $\Ai'$ has pure
slope $1/2$ at infinity, the restriction of $V$ to
$I_\infty$ has pure slope $3/2$.

We claim that $P_\infty$ has image $Q_8$.
Indeed, since $G_{\geom}=Q_8$ is a $2$-group, the
images of $I_\infty$ and $P_\infty$ coincide.
If this image were a proper subgroup of $Q_8$,
it would be abelian, so $V|_{I_\infty}$ would split
into rank-one characters. By Hasse--Arf, their
breaks would be integers, contradicting the pure
slope $3/2$.

Consequently, each nontrivial character $1_i$ of
$Q_8$ restricts to a nontrivial rank-one character
of $P_\infty$. Since $1_i$ occurs in $V^{\otimes2}$,
its break is at most $3/2$, and hence equals $1$
by Hasse--Arf. Therefore
\[
\Swan_\infty(V)=3,\qquad
\Swan_\infty(1_i)=1\quad(i=1,2,3),\qquad
\Swan_\infty(1_0)=0.
\]

The recurrence
$\Sym^{j+1}V=V\otimes\Sym^jV-\Sym^{j-1}V$ in the
representation ring, together with
$V\otimes V=1_0\oplus W$ and $V\otimes W=3V$, yields
\begin{equation}\label{eq:Q_8_decom_Sym}
\begin{aligned}
\Sym^{2m+1}V&\simeq(m+1)V,\\
\Sym^{2m}V&\simeq a_m1_0\oplus b_m W,\\
 b_m&=\left\lfloor\frac{m+1}{2}\right\rfloor,
 \qquad a_m=2m+1-3b_m.
\end{aligned}
\end{equation}
Additivity of Swan conductors now gives the two formulas.
\end{proof}

\begin{lemma}\label{lem:swan_symk}
In characteristic $2$, the Swan conductors of the symmetric powers
of $\Ai'$ at infinity are
\[
\Swan_\infty(\Sym^{2m+1}\Ai')=m+1,\qquad
\Swan_\infty(\Sym^{2m}\Ai')
      =\left\lfloor\frac{m+1}{2}\right\rfloor.
\]
\end{lemma}
\begin{proof}
The cubic Kummer pullback is tame at infinity and
$[3]^*\Sym^k\Ai'\simeq\Sym^k\Ai|_{\Gm}$; hence
$\Swan_\infty(\Sym^k\Ai)=3\Swan_\infty(\Sym^k\Ai')$.
Apply Lemma~\ref{lem:swan_symk_Q8}.
\end{proof}

\begin{proposition}\label{prop:dim_H1_p2}
Let $p=2$. Then
\[
\sum_{k\ge0}h_c^2(\Gm,\Sym^k\Ai')t^k
=\frac{1+t^{12}}{(1-t^6)(1-t^8)},
\]
and
\[
h_c^1(\Gm,\Sym^k\Ai')=
\begin{cases}
(k+1)/2,&k\text{ odd},\\
\lfloor(k+2)/4\rfloor+h_c^2(\Gm,\Sym^k\Ai'),&k\text{ even}.
\end{cases}
\]
\end{proposition}
\begin{proof}
The representation on the generic stalk of $\Ai'$ is the faithful
symplectic two-dimensional representation $2_0$ of
$\SL(2,\mathbb F_3)$; the relevant character values appear in
Table~\ref{tab:char_table_SL2F3}.
\begin{table}[ht]
    \centering
    \renewcommand{\arraystretch}{1.2}
    \begin{tabular}{c|ccccccc}
        \hline
        Class & $1A$ & $2A$ & $4A$ & $3A$ & $3B$ & $6A$ & $6B$ \\
        Size & $1$ & $1$ & $6$ & $4$ & $4$ & $4$ & $4$ \\
        \hline
        $1_{0}$ & $1$ & $1$ & $1$ & $1$ & $1$ & $1$ & $1$ \\
        $1_{1}$ & $1$ & $1$ & $1$ & $\omega$ & $\omega^{2}$ & $\omega$ & $\omega^{2}$ \\
        $1_{2}$ & $1$ & $1$ & $1$ & $\omega^{2}$ & $\omega$ & $\omega^{2}$ & $\omega$ \\
        $W_{3}$ & $3$ & $3$ & $-1$ & $0$ & $0$ & $0$ & $0$ \\
        $2_{0}$ & $2$ & $-2$ & $0$ & $-1$ & $-1$ & $1$ & $1$ \\
        $2_{1}$ & $2$ & $-2$ & $0$ & $-\omega$ & $-\omega^{2}$ & $\omega$ & $\omega^{2}$ \\
        $2_{2}$ & $2$ & $-2$ & $0$ & $-\omega^{2}$ & $-\omega$ & $\omega^{2}$ & $\omega$ \\
        \hline
    \end{tabular}
    \caption{Character table of $\SL(2,\mathbb{F}_{3})$, where $\omega = e^{2\pi i / 3}$.}
    \label{tab:char_table_SL2F3}
\end{table}
Molien's formula gives
\[
\begin{aligned}
\sum_{k\ge0}h_c^2(\Gm,\Sym^k\Ai')t^k
&=\frac1{24}\left(\frac1{(1-t)^2}+\frac1{(1+t)^2}
 +\frac6{1+t^2}+\frac8{1+t+t^2}+\frac8{1-t+t^2}\right)\\
&=\frac{1+t^{12}}{(1-t^6)(1-t^8)}.
\end{aligned}
\]
This series is even, so $h_c^2=0$ for odd $k$.
The remaining assertion follows from
\eqref{eq:airy_GOS_dimensions} and Lemma~\ref{lem:swan_symk}.
\end{proof}

\begin{proposition}\label{prop:dim_H1_p2_Ai}
Let $p=2$. Then
\[
\sum_{k\ge0}h_c^2(\mathbb A^1,\Sym^k\Ai)t^k
  =\frac{1+t^6}{(1-t^4)^2},
\]
and
\[
h_c^1(\mathbb A^1,\Sym^k\Ai)
=\begin{cases}
(k+1)/2,&k\text{ odd},\\
0,&k\text{ even}.
\end{cases}
\]
\end{proposition}
\begin{proof}
For the faithful two-dimensional representation $V$ of $Q_8$,
Molien's formula and Table~\ref{tab:char_table_Q8} give
\[
\sum_{k\ge0}h_c^2(\mathbb A^1,\Sym^k\Ai)t^k
=\frac18\left(\frac1{(1-t)^2}+\frac1{(1+t)^2}
+\frac6{1+t^2}\right)
=\frac{1+t^6}{(1-t^4)^2}.
\]
The odd part vanishes. For $k=2m$, the decomposition \eqref{eq:Q_8_decom_Sym}
in Lemma~\ref{lem:swan_symk_Q8} gives
$h_c^2=a_m$ and
$\Swan_\infty=3b_m$, with $a_m+3b_m=2m+1$.
Equation~\eqref{eq:airy_GOS_dimensions} then yields $h_c^1=0$.
For $k=2m+1$ it gives
$h_c^1=3(m+1)-(2m+2)=m+1$.
\end{proof}

\subsection{Motives associated with Airy moments}
\label{subsec:airy_moments_motive}
Retain the notation and assumptions of \S\ref{subsec:airy_sheaves}.
For $k\ge2$, set
\[
 g_k(y)=\sum_{i=1}^k\left(\frac{y_i^3}{3}-y_i\right),
 \qquad A'=\{g_k=0\}\subset\mathbb A_y^k.
\]
The group $S_k$ acts by permuting the coordinates and $\mu_2$
acts by $y\mapsto-y$. Put
$\chi=\operatorname{sgn}\boxtimes\mathbf1$, and write $V^\chi$
for the $\chi$-isotypic part of an $S_k\times\mu_2$-representation $V$.
When $p=2$, the involution is trivial,
so $V^\chi$ is simply the $S_k$-sign part.

\begin{proposition}\label{prop:Aiprime_cohomology_hypersurface}
For $k\ge2$ and $j=1,2$, there are Frobenius-equivariant isomorphisms
\[
 H_c^j(\Gm,\Sym^k\Ai')
 \simeq H_c^{k+j-2}(A',\Qlbar)^\chi(-1).
\]
\end{proposition}

\begin{proof}
Let
\[
 F(z',x)=\sum_{i=1}^k\left(\frac{x_i^3}{3z'}-x_i\right),
 \qquad x=(x_1,\ldots,x_k).
\]
Since $R\pi_!\mathcal L_{\psi(x^3/(3z')-x)}\simeq\Ai'[-1]$,
the relative K\"unneth formula gives
\[
 R\pi_{k,!}\mathcal L_{\psi(F)}\simeq(\Ai')^{\otimes k}[-k],
 \qquad \pi_k:\Gm_{z'}\times\mathbb A_x^k\to\Gm_{z'}.
\]
Permuting the $x_i$ corresponds to the signed permutation
of the factors on $(\Ai')^{\otimes k}$, because each factor
occurs in cohomological degree $1$. Hence
\begin{equation}\label{eq:aiprime-kunneth-geom}
 H_c^j(\Gm_{z'},\Sym^k\Ai')
 \simeq H_c^{k+j}(\Gm_{z'}\times\mathbb A_x^k,
               \mathcal L_{\psi(F)})^{\operatorname{sgn}}.
\end{equation}

Under the quadratic substitution $z'=w^2$, $x_i=wy_i$,
we have $F(w^2,wy)=w\,g_k(y)$. If $p\ne2$, taking invariants
under the deck involution, which becomes
$(w,y)\mapsto(-w,-y)$, recovers the original cohomology.
If $p=2$, the squaring map is a universal homeomorphism;
topological invariance of \'etale cohomology gives the same
identification, without deck invariants. Thus, with
$U=\Gm_w\times\mathbb A_y^k$,
\begin{equation}\label{eq:aiprime-kummer-geom}
 H_c^j(\Gm_{z'},\Sym^k\Ai')
 \simeq H_c^{k+j}(U,\mathcal L_{\psi(wg_k(y))})^\chi.
\end{equation}

Now let $X=\mathbb A_w^1\times\mathbb A_y^k$,
$p:X\to\mathbb A_y^k$ be the projection, and
$i:A'\hookrightarrow\mathbb A_y^k$ the inclusion.
By base change along $g_k:\mathbb A_y^k\to\mathbb A^1_s$,
\[
 Rp_!\mathcal L_{\psi(wg_k(y))}
 \simeq g_k^*R(\operatorname{pr}_s)_!\mathcal L_{\psi(ws)}
 \simeq i_*\Qlbar(-1)[-2].
\]
Indeed, $R\Gamma_c(\mathbb A_w^1,\mathcal L_{\psi(ws)})$
vanishes for $s\ne0$ and equals $\Qlbar(-1)[-2]$ at $s=0$.
Consequently,
\begin{equation}\label{eq:aiprime-integrate-w}
 H_c^a(X,\mathcal L_{\psi(wg_k(y))})
 \simeq H_c^{a-2}(A',\Qlbar)(-1).
\end{equation}
This is $S_k\times\mu_2$-equivariant: the involution
$(w,y)\mapsto(-w,-y)$ acts trivially on
$H_c^2(\mathbb A_w^1,\Qlbar)$.

Finally, $X\setminus U=Z=\{0\}\times\mathbb A_y^k$,
and the phase vanishes on $Z$. The localization triangle
\[
 R\Gamma_c(U,\mathcal L_{\psi(wg_k(y))})
 \longrightarrow R\Gamma_c(X,\mathcal L_{\psi(wg_k(y))})
 \longrightarrow R\Gamma_c(Z,\Qlbar)\xrightarrow{+1}
\]
has $R\Gamma_c(Z,\Qlbar)\simeq\Qlbar(-k)[-2k]$.
Since $S_k$ acts trivially on this top compactly supported
cohomology, its $\chi$-isotypic part vanishes for $k\ge2$.
Thus localization identifies the $\chi$-parts of the
cohomology of $U$ and $X$. Combining
\eqref{eq:aiprime-kummer-geom} with
\eqref{eq:aiprime-integrate-w} proves the claim.
All the constructions are Frobenius-equivariant.
\end{proof}

The action
\[
 (z,x)\longmapsto(\zeta^2z,\zeta x),\qquad \zeta\in\mu_3,
\]
preserves $x^3/3-zx$ and induces a $\mu_3$-equivariant
structure on $\Ai$. For a $\Qlbar$-representation $V$ of
$\mu_3$ with action $\rho_V$, put
\[
 V^{\mu_3}=e_1V,\qquad V_{\ne1}=(1-e_1)V,
 \qquad e_1=\frac13\sum_{\zeta\in\mu_3}\rho_V(\zeta).
\]
Thus $V_{\ne1}$ is the sum of the two nontrivial character
spaces. 
Whenever $V$ carries a compatible Frobenius action,
both $V^{\mu_3}$ and $V_{\ne1}$ are Frobenius-stable,
even when $q\equiv2\pmod3$ and Frobenius exchanges
the two nontrivial character spaces.

\begin{proposition}\label{prop:Ai_cohomology_hypersurface}
For $k\ge2$, there is an exact sequence
\[
\begin{aligned}
0\longrightarrow
 \left(\Sym^k H_c^1(\mathbb A^1,
 \mathcal L_{\psi(x^3/3)})\right)^{\mu_3}
 &\longrightarrow H_c^{k-1}(A',\Qlbar)^\chi(-1)\longrightarrow H_c^1(\mathbb A^1,\Sym^k\Ai)^{\mu_3}
 \longrightarrow0.
\end{aligned}
\]
Moreover,
\[
 H_c^2(\mathbb A^1,\Sym^k\Ai)^{\mu_3}
 \simeq H_c^k(A',\Qlbar)^\chi(-1).
\]
\end{proposition}

\begin{proof}
By Proposition~\ref{prop:airy_cubic_pullback},
$[3]^*\Ai'\simeq\Ai|_{\Gm_z}$ for the cubic Kummer covering
$[3]:\Gm_z\to\Gm_{z'}$, $z'=z^3$.
The pullback carries the deck action $z\mapsto\zeta^2z$,
while $\Ai$ carries the geometric action
$(z,x)\mapsto(\zeta^2z,\zeta x)$. These agree under the
change of variables $(z,x)\mapsto(z^3,zx)$, which is
invariant under the latter action. Thus the isomorphism is
$\mu_3$-equivariant.
Finite \'etale descent and
Proposition~\ref{prop:Aiprime_cohomology_hypersurface}
now give, for $r=1,2$,
\begin{equation}\label{eq:airy-mu3-punctured-4-2}
\begin{aligned}
H_c^r(\Gm_z,\Sym^k\Ai)^{\mu_3}
&\simeq H_c^r(\Gm_{z'},\Sym^k\Ai')\simeq H_c^{k+r-2}(A',\Qlbar)^\chi(-1).
\end{aligned}
\end{equation}

Let
\[
\mathbb E_{3,0}:=
H_c^1(\mathbb A_x^1,\mathcal L_{\psi(x^3/3)}).
\]
By base change, $\Ai_0\simeq\mathbb E_{3,0}$.
The cubic covering $x\mapsto x^3/3$ gives
$\mathbb E_{3,0}=\rho\oplus\rho^{-1}$,
where $\rho$ is a nontrivial $\mu_3$-character.

The $\mu_3$-equivariant localization sequence for
$\Gm_z\subset\mathbb A_z^1$, together with
$H_c^0(\mathbb A_z^1,\Sym^k\Ai)=0$, yields
\[
\begin{aligned}
0\longrightarrow(\Sym^k \mathbb{E}_{3,0})^{\mu_3}
&\longrightarrow H_c^1(\Gm_z,\Sym^k\Ai)^{\mu_3}\longrightarrow H_c^1(\mathbb A_z^1,\Sym^k\Ai)^{\mu_3}
\longrightarrow0,
\end{aligned}
\]
and
\[
H_c^2(\Gm_z,\Sym^k\Ai)^{\mu_3}
\simeq H_c^2(\mathbb A_z^1,\Sym^k\Ai)^{\mu_3}.
\]
Substituting \eqref{eq:airy-mu3-punctured-4-2}
proves both assertions.
\end{proof}

To describe the nontrivial $\mu_3$-part, we follow the
cyclic-cover construction in \cite[Proposition~5.10]{SY23}.
The quadratic substitution $t=s^2$, $x_i=sy_i$ transforms
the Airy phase into $s^3g_k(y)$. Over
$U:=\mathbb A_y^k\setminus A'$, adjoining a cube root
$z^3=g_k(y)$ makes the phase equal to $(sz)^3$.
This leads to the finite \'etale cubic cover
\[
 \pi:A''=\{(z,y)\in\Gm_z\times\mathbb A_y^k:
                 z^3=g_k(y)\}\longrightarrow U,
 \qquad (z,y)\longmapsto y.
\]
Here $S_k$ permutes $y$, $\mu_2$ acts by
$(z,y)\mapsto(-z,-y)$, and the deck group $\mu_3$ acts
by $(z,y)\mapsto(\zeta z,y)$. Define
\[
 \mathbb E_3=H_c^1(\mathbb A_w^1,\mathcal L_{\psi(w^3)}).
\]
Under $w\mapsto\zeta w$, the two nontrivial $\mu_3$-characters
occur in $\mathbb E_3$ with multiplicity one each;
$S_k\times\mu_2$ acts trivially on $\mathbb E_3$.

\begin{proposition}
\label{prop:Ai_cohomology_non_inv_part_A''}
For $k\ge2$, there is a canonical isomorphism
of $\mu_3$-representations
\[
H_c^1(\Gm,\Sym^k\Ai)_{\ne1}
\simeq
\left(
\mathbb E_3\otimes H_c^k(A'',\Qlbar)^\chi
\right)^{\mu_3}.
\]
Here the invariants are taken under the diagonal
$\mu_3$-action, and the $\mu_3$-representation
structure on the right is inherited from
the $\mathbb E_3$ factor.
\end{proposition}

\begin{proof}
As in the proof of
Proposition~\ref{prop:Aiprime_cohomology_hypersurface},
the relative K\"unneth formula and the quadratic substitution
above give
\[
 H_c^1(\Gm,\Sym^k\Ai)
 \simeq H_c^{k+1}
 \bigl(\Gm_s\times\mathbb A_y^k,
       \mathcal L_{\psi(s^3g_k(y))}\bigr)^\chi.
\]
The original $\mu_3$-action becomes
$(s,y)\mapsto(\zeta s,y)$.
On $\Gm_s\times A'$ the phase is zero and this $\mu_3$
acts trivially on cohomology. Hence localization gives
\[
 H_c^1(\Gm,\Sym^k\Ai)_{\ne1}
 \simeq H_c^{k+1}\bigl(\Gm_s\times U,
       \mathcal L_{\psi(s^3g_k(y))}\bigr)^\chi_{\ne1}.
\]

Pulling back along $\pi:A''\to U$ and setting $w=sz$
identifies $\Gm_s\times A''$ with $\Gm_w\times A''$,
where $s^3g_k(y)=w^3$. On this product there are two
commuting $\mu_3$-actions:
\[
\begin{aligned}
 \mu_3^{\mathrm{deck}}:&\quad
 (w,z,y)\longmapsto(\zeta w,\zeta z,y),\\
 \mu_3^{\mathrm{orig}}:&\quad
 (w,z,y)\longmapsto(\zeta w,z,y).
\end{aligned}
\]
The first is used for descent, while the second is induced
by the original action on $\Ai$. Moreover,
$S_k\times\mu_2$ acts only on $A''$, since
$(s,z,y)\mapsto(-s,-z,-y)$ fixes $w$.
Finite \'etale descent and the K\"unneth formula now give an isomorphism
compatible with all these actions:
\[
\begin{aligned}
&R\Gamma_c\bigl(\Gm_s\times U,
        \mathcal L_{\psi(s^3g_k(y))}\bigr)\simeq\left(
 R\Gamma_c(\Gm_w,\mathcal L_{\psi(w^3)})
 \overset{L}{\otimes}R\Gamma_c(A'',\Qlbar)
\right)^{\mu_3^{\mathrm{deck}}}.
\end{aligned}
\]

Finally, localization at $w=0$ gives
\[
0\longrightarrow\Qlbar
\longrightarrow H_c^1(\Gm_w,\mathcal L_{\psi(w^3)})
\longrightarrow\mathbb E_3\longrightarrow0.
\]
The first term is $\mu_3^{\mathrm{orig}}$-trivial, whereas
$\mathbb E_3$ has only nontrivial characters. As the other
compactly supported cohomology groups vanish,
\[
R\Gamma_c(\Gm_w,\mathcal L_{\psi(w^3)})
_{\mu_3^{\mathrm{orig}},\ne1}
\simeq\mathbb E_3[-1].
\]
Taking the $\chi$-isotypic part and the
$\mu_3^{\mathrm{deck}}$-invariants yields the assertion,
where the original $\mu_3$-action is induced by its
action on $\mathbb E_3$.
\end{proof}

\begin{remark}\label{Rmk:etale_in_p_same_as_de_Rham}
The preceding three propositions admit analogous statements
in the de Rham setting over $\mathbb C$.
Indeed, replacing the Artin--Schreier sheaves by exponential
connections and the $\ell$-adic Airy sheaves by their
corresponding Airy connections, the same arguments apply
to compactly supported de Rham cohomology. 
We will use the same notation for these corresponding
de Rham realizations in characteristic zero.
\end{remark}

For a variety $X$ and a coefficient sheaf $\mathcal F$, we write
\[
H^i_{\mathrm{mid}}(X,\mathcal F)
:=\operatorname{im}\bigl(
H_c^i(X,\mathcal F)\longrightarrow H^i(X,\mathcal F)
\bigr).
\]
Let $e_{\ne1}$ denote the projector onto the sum of the
two nontrivial $\mu_3$-character spaces. We introduce the following notation for the middle
cohomology realizations associated with Airy moments:
\begin{equation}\label{eq:airy_motives_middle}
\begin{aligned}
 M_k&:=H^1_{\mathrm{mid}}(\mathbb A_z^1,\Sym^k\Ai),\\
 M'_k&:=M_k^{\mu_3},\qquad
 M''_k:=e_{\ne1}M_k,\\
 M_k^{\flat}&:=H^k_{\mathrm{mid}}(A'',\Qlbar)^\chi.
\end{aligned}
\end{equation}
Here $\mu_3$ acts on $M_k$ through the equivariant
structure of $\Ai$, and on $M_k^{\flat}$ through the
deck transformations of $A''$. In particular, $M_k=M'_k\oplus M''_k$.
As shown in
Proposition~\ref{prop:airy_compact_to_middle_general},
the $\mu_3$-invariant part of $M_k^{\flat}$ vanishes,
so $M_k^{\flat}=e_{\ne1}M_k^{\flat}$.

The middle Airy realizations $M_k,M'_k,M''_k$
are pure of weight $k+1$, while $M_k^\flat$
is pure of weight $k$. In positive characteristic,
this follows from the purity of $\Ai$ and the
weight bounds for middle cohomology, noting
that $A''$ is smooth.
The characteristic-zero statements follow
from the corresponding Hodge-theoretic
constructions (see \cite{SY23}).

For the subsequent specialization arguments, we also
introduce the auxiliary compactly supported realizations
\begin{equation}\label{eq:airy_motives_compact_support}
\begin{aligned}
 C'_k&:=H_c^{k-1}(A',\Qlbar)^\chi(-1),\\
 C''_k&:=H_c^k(A'',\Qlbar)^\chi,\qquad
 N_k:=e_{\ne1}C''_k.
\end{aligned}
\end{equation}
Unlike the middle realizations, these carry the
boundary contributions and need not be pure.

\begin{proposition}[The non-classical middle realization]
\label{prop:airy_nonclassical_middle_identification}
For $k\ge2$, there is a natural $\mu_3$-equivariant isomorphism
\begin{equation}\label{eq:airy_nonclassical_middle_identification}
M''_k\simeq(\mathbb E_3\otimes M_k^{\flat})^{\mu_3},
\end{equation}
where the invariants are taken under the diagonal
$\mu_3$-action, while the original Airy $\mu_3$-action
on the right is induced by its action on $\mathbb E_3$ alone. In positive characteristic, this isomorphism is
compatible with the Galois action. The analogous statement holds for the de Rham realizations
in characteristic zero.
\end{proposition}

\begin{proof}
The construction of
Proposition~\ref{prop:Ai_cohomology_non_inv_part_A''}
is compatible with the natural maps from compactly
supported to ordinary cohomology.
Indeed, for the projection
$q:\mathbb A_x^1\times\Gm_t\to\Gm_t$,
the Artin--Schreier sheaf defining $\Ai$
has no inertia invariants at $x=\infty$.
Hence $Rq_!\to Rq_*$ is an isomorphism, and
the same holds for the $k$-fold product
by the relative K\"unneth formula. Moreover, the localization terms supported on
$\Gm_s\times A'$ are trivial under the original
$\mu_3$-action, both for compactly supported
cohomology and cohomology with supports.

On the other hand, since
$\mathcal L_{\psi(w^3)}$ has no inertia invariants
at $w=\infty$, the natural map
\[
H_c^1(\mathbb A_w^1,\mathcal L_{\psi(w^3)})
\xrightarrow{\sim}
H^1(\mathbb A_w^1,\mathcal L_{\psi(w^3)})
\]
is an isomorphism. The localization terms at
$w=0$ are $\mu_3$-trivial. Consequently, on
the nontrivial $\mu_3$-part, the cubic-cover
construction identifies $H_c^1\to H^1$ with
\[
\left(
\mathrm{id}_{\mathbb E_3}\otimes
[H_c^k(A'')^\chi\to H^k(A'')^\chi]
\right)^{\mu_3}.
\]

Finally, since $\Sym^k\Ai$ is lisse at $0$,
localization gives a surjection on $H_c^1$
and an injection on $H^1$ when passing from
$\Gm$ to $\mathbb A^1$. The canonical comparison
map factors through these maps, so the middle
image is unchanged. Taking images in the preceding
identification proves
\eqref{eq:airy_nonclassical_middle_identification}.
\end{proof}

\begin{proposition}
\label{prop:airy_Mdoubleprime_mu3_trivial}
Let $k\ge2$ and let $F$ be a field of characteristic different from
$3$, with $\ell\ne\operatorname{char}F$. Then, on geometric $\ell$-adic
realizations, there is a canonical isomorphism, compatible with the
Galois action of $F$ and with $S_k\times\mu_2$,
\[
 H_c^k(A''_{\overline F},\Qlbar)^{\mu_3}
 \simeq H_c^{k-1}(A'_{\overline F},\Qlbar).
\]
Consequently, after taking the $\chi$-isotypic
component,
\begin{equation}\label{eq:airy_Mdoubleprime_trivial_is_Mprime}
(C''_k)^{\mu_3}\simeq C'_k(1).
\end{equation}

\end{proposition}

\begin{proof}
Put $U=\mathbb A_F^k\setminus A'$. The projection
$\pi:A''\to U$ is a finite \'etale $\mu_3$-torsor.
By finite \'etale descent, we have an
$S_k\times\mu_2$- and Galois-equivariant isomorphism
\[
R\Gamma_c(U_{\overline F},\Qlbar)
\simeq R\Gamma_c(A''_{\overline F},\Qlbar)^{\mu_3}.
\]
The localization sequence for
$A'\hookrightarrow\mathbb A^k\hookleftarrow U$
gives
\[
H_c^{k-1}(\mathbb A^k_{\overline F},\Qlbar)
\longrightarrow H_c^{k-1}(A'_{\overline F},\Qlbar)
\longrightarrow H_c^k(U_{\overline F},\Qlbar)
\longrightarrow H_c^k(\mathbb A^k_{\overline F},\Qlbar).
\]
Since $H_c^i(\mathbb A^k_{\overline F},\Qlbar)=0$
for $i\ne2k$, both outer terms vanish for $k\ge2$.
Combining the two isomorphisms and taking the
$\chi$-isotypic component proves the assertion.
The Tate twist follows from the definition of $C'_k$.
\end{proof}

\begin{proposition}
\label{prop:airy_compact_to_middle_general}
Let $k\ge2$, and let $F$ be a field of characteristic
different from $3$, with $\ell\ne\operatorname{char}F$.
Then
\[
H_{\mathrm{mid}}^k(A''_{\overline F},\Qlbar)^{\mu_3}=0.
\]
In particular,
\begin{equation}\label{eq:airy_middle_no_mu3_trivial}
M_k^{\flat}
=H_{\mathrm{mid}}^k(A''_{\overline F},\Qlbar)^\chi
=e_{\ne1}H_{\mathrm{mid}}^k(A''_{\overline F},\Qlbar)^\chi.
\end{equation}
\end{proposition}

\begin{proof}
Put $U=\mathbb A_F^k\setminus A'$.
Since $A''\to U$ is a finite \'etale $\mu_3$-torsor,
descent identifies the natural map
\[
H_c^k(A''_{\overline F},\Qlbar)^{\mu_3}
\longrightarrow
H^k(A''_{\overline F},\Qlbar)^{\mu_3}
\]
with $H_c^k(U_{\overline F})\to H^k(U_{\overline F})$.

The latter map factors through the cohomology
of the ambient affine space:
\[
H_c^k(U_{\overline F})
\longrightarrow H_c^k(\mathbb A_{\overline F}^k)
\longrightarrow H^k(\mathbb A_{\overline F}^k)
\longrightarrow H^k(U_{\overline F}).
\]
Since $H_c^k(\mathbb A_{\overline F}^k)=0$
for $k\ge2$, this map vanishes.
Consequently,
\[
H_{\mathrm{mid}}^k(A''_{\overline F})^{\mu_3}=0,
\]
which proves the assertion after taking the
$\chi$-isotypic component.
\end{proof}

\begin{corollary}[Compactly supported and middle realizations]
\label{cor:airy_middle_not_compact_decomposition}
For $k\ge2$, the $\mu_3$-isotypic decomposition
of $C''_k$ gives a canonical isomorphism
\[
C''_k
\simeq (C''_k)^{\mu_3}\oplus N_k
\simeq C'_k(1)\oplus N_k.
\]
Moreover, the natural map to middle cohomology
induces a surjection
\[
N_k\twoheadrightarrow M_k^{\flat},
\]
while the $\mu_3$-invariant summand $C'_k(1)$
maps to zero.
\end{corollary}

\begin{proof}
The first assertion follows from
Proposition~\ref{prop:airy_Mdoubleprime_mu3_trivial}.
The second follows from
Proposition~\ref{prop:airy_compact_to_middle_general},
since $M_k^{\flat}$ has no $\mu_3$-invariants.
\end{proof}

\begin{corollary}
\label{cor:airy_Mdoubleprime_full_vs_nontrivial}
Put $b_k:=\dim(\Sym^k\mathbb E_{3,0})_{\ne1}$.
For $k\ge2$, there is a natural exact sequence
\[
\begin{aligned}
0\longrightarrow(\Sym^k\mathbb E_{3,0})_{\ne1}
&\longrightarrow(\mathbb E_3\otimes N_k)^{\mu_3}\longrightarrow
H_c^1(\mathbb A^1,\Sym^k\Ai)_{\ne1}
\longrightarrow0,
\end{aligned}
\]
where the invariants on the middle term are taken
with respect to the diagonal deck action.
In particular,
\[
\dim N_k
=\dim H_c^1(\mathbb A^1,\Sym^k\Ai)_{\ne1}+b_k,
\]
where
\[
b_k=\#\{0\le j\le k:
2k-j\not\equiv0\pmod3\}.
\]
\end{corollary}

\begin{proof}
By Proposition~\ref{prop:Ai_cohomology_non_inv_part_A''},
the nontrivial $\mu_3$-part of
$H_c^1(\Gm,\Sym^k\Ai)$ is identified with
$(\mathbb E_3\otimes N_k)^{\mu_3}$,
since $\mathbb E_3$ has no trivial $\mu_3$-character.
The exact sequence then follows from localization
at the Airy parameter $0$, using
$\Ai_0\simeq\mathbb E_{3,0}$.

Finally, $\mathbb E_{3,0}$ contains each nontrivial
$\mu_3$-character exactly once, giving the formula
for $b_k$. Since the same holds for $\mathbb E_3$,
taking diagonal invariants in
$\mathbb E_3\otimes N_k$ preserves the dimension
of $N_k$, proving the dimension formula.
\end{proof}

\begin{lemma}[Dimensions of the generic realizations]
\label{lem:dim_over_Qp_and_de_Rham}
For $k\ge2$, put
\[
\begin{aligned}
r_k&=
\begin{cases}
(k+1)/2,&k\text{ odd},\\
k/2-1,&k\text{ even},
\end{cases}\\
d_k&=\left\lfloor\frac{k}{3}\right\rfloor
+\delta_{1+2\mathbb Z}(k)
-\delta_{1+3\mathbb Z}(k),\\
\epsilon_k&=\delta_{4\mathbb Z}(k),
\qquad
n_k=\left\lfloor\frac{k+1}{2}\right\rfloor.
\end{aligned}
\]
Here $\delta_S$ denotes the indicator function of $S$.
Then the dimensions of the characteristic-zero
realizations are
\begin{equation}\label{eq:airy_generic_dimensions_corrected}
\begin{aligned}
\dim M_k&=r_k-\epsilon_k,\\
\dim M'_k&=r_k-d_k-\epsilon_k,\\
\dim M''_k&=\dim M_k^\flat=d_k,\\
\dim C'_k&=\dim(C''_k)^{\mu_3}=n_k,\\
\dim N_k&=d_k+b_k=2n_k,\\
\dim C''_k&=3n_k.
\end{aligned}
\end{equation}
Here $b_k$ is as in
Corollary~\ref{cor:airy_Mdoubleprime_full_vs_nontrivial}.
The formulas for the geometric realizations
$C'_k,C''_k,N_k$ and $M_k^\flat$
also hold for their geometric $\ell$-adic realizations
over $\overline{\mathbb Q}_p$, for every prime $p$.
\end{lemma}

\begin{proof}
By \cite[Corollary~7.6 and Proposition~7.9]{SY23},
we have
\[
\dim H^1_{\mathrm{dR}}(\mathbb A^1,\Sym^k\Ai)=r_k.
\]
This cohomology group admits a basis
$\{\omega_i\}_{1\le i\le r_k}$.
The $\mu_3$-action sends $\omega_i$ to
$\zeta^{k+2i}\omega_i$, so $d_k$ counts
the basis elements with nontrivial character.
When $4\mid k$, the quotient by middle cohomology
is generated by $\omega_{k/4}$, whose character
is trivial. Consequently,
\[
\dim M_k=r_k-\epsilon_k,\qquad
\dim M'_k=r_k-d_k-\epsilon_k,\qquad
\dim M''_k=d_k.
\]
Since $\mathbb E_3$ contains each nontrivial
$\mu_3$-character once,
Propositions~\ref{prop:airy_nonclassical_middle_identification}
and \ref{prop:airy_compact_to_middle_general}
give $\dim M_k^\flat=d_k$.

Next, write
$\mathbb E_{3,0}=\rho\oplus\rho^{-1}$,
where $\rho$ is a nontrivial $\mu_3$-character.
The monomials in $\Sym^k\mathbb E_{3,0}$
have characters $\rho^{2k-j}$, for $0\le j\le k$.
Hence the multiplicity of the trivial character is
\[
\begin{aligned}
a_k
&:=\dim(\Sym^k\mathbb E_{3,0})^{\mu_3}\\
&=\#\{0\le j\le k:2k-j\equiv0\pmod3\}\\
&=\left\lfloor\frac{k}{3}\right\rfloor
+1-\delta_{1+3\mathbb Z}(k).
\end{aligned}
\]
In particular,
\[
a_k+b_k=k+1,\qquad
a_k-d_k=1-\delta_{1+2\mathbb Z}(k).
\]

By self-duality and Poincar\'e duality,
\[
\dim H_c^1(\mathbb A^1,\Sym^k\Ai)^{\mu_3}
=r_k-d_k.
\]
Proposition~\ref{prop:Ai_cohomology_hypersurface}
therefore gives
\[
\dim C'_k
=a_k+r_k-d_k
=r_k+1-\delta_{1+2\mathbb Z}(k)
=n_k.
\]

Finally, Corollaries~\ref{cor:airy_Mdoubleprime_full_vs_nontrivial}
and \ref{cor:airy_middle_not_compact_decomposition}
give
\[
\begin{aligned}
\dim N_k
&=d_k+b_k
=k+1-(a_k-d_k)=2n_k,\\
\dim C''_k
&=\dim C'_k+\dim N_k=3n_k.
\end{aligned}
\]
The comparison between Betti and geometric
$\ell$-adic cohomology preserves dimensions
and finite-group projectors.
\end{proof}

\subsection{Galois representations of motives associated with Airy moments}
\label{subsec:airy_arithmetic_revised}
We compare the geometric realizations
$C'_k,C''_k,N_k$ and $M_k^\flat$ introduced in
\S\ref{subsec:airy_moments_motive}
over $\overline{\mathbb Q}_p$ and $\overline{\mathbb F}_p$.
To apply the arithmetic Picard--Lefschetz formula,
we first compactify $A'$ and $A''$, identify their
singularities, and determine the contributions
of vanishing cycles.

Unless otherwise stated, let $p\ge5$, $\ell\ne p$,
$k\ge2$, and $G=S_k\times\mu_2$, with
$\chi=\operatorname{sgn}\boxtimes\mathbf1$.
For $F=\Qp$ or $\Fp$, put
\[
\begin{aligned}
 C'_{k,F}
 &=H_c^{k-1}(A'_{\overline F},\Qlbar)^\chi(-1),\\
 C''_{k,F}
 &=H_c^k(A''_{\overline F},\Qlbar)^\chi,\\
 N_{k,F}
 &=e_{\ne1}C''_{k,F},\\
 M^\flat_{k,F}
 &=H_{\mathrm{mid}}^k(A''_{\overline F},\Qlbar)^\chi.
\end{aligned}
\]
By Corollary~\ref{cor:airy_middle_not_compact_decomposition},
we have
\[
C''_{k,F}\simeq C'_{k,F}(1)\oplus N_{k,F},
\qquad
N_{k,F}\twoheadrightarrow M^\flat_{k,F}.
\]
These maps are compatible with the local Galois
actions. In particular, $e_{\ne1}$ is Galois-stable
even when Frobenius interchanges the two nontrivial
cubic characters.

\paragraph{Compactifications and singularities.}
Recall that
$g_k(y)=\sum_{i=1}^k(y_i^3/3-y_i)$.
Over $\mathbb Z_p$, consider the projective
compactifications
\[
\begin{aligned}
T:=\overline{A'}
&=\left\{
\sum_{i=1}^k(y_i^3/3-y_iw^2)=0
\right\}
\subset\mathbb P^k_{[w:y]},\\
Y:=\overline{A''}
&=\left\{
z^3-\sum_{i=1}^k(y_i^3/3-y_iw^2)=0
\right\}
\subset\mathbb P^{k+1}_{[w:z:y]}.
\end{aligned}
\]
The boundary $D'=T\cap\{w=0\}$ is a smooth
Fermat cubic. The boundary of $A''$ in $Y$ is
$D=D_0\cup D_\infty$, where
\[
D_0=Y\cap\{z=0\}\simeq T,
\qquad
D_\infty=Y\cap\{w=0\}.
\]
Both $D_\infty$ and $D_0\cap D_\infty$ are smooth
Fermat hypersurfaces with good reduction.
By the Jacobian criterion, $T$ and $Y$ are
smooth over $\mathbb Z_p$ along $w=0$.
Thus the vanishing-cycle complexes of the
constant sheaf are supported at the singular
points of the special fibres in the affine
chart $w=1$.

The critical points of $g_k$ are precisely
the sign vectors
$\boldsymbol\varepsilon
=(\varepsilon_1,\ldots,\varepsilon_k)\in\{\pm1\}^k$.
Since
\[
g_k(\boldsymbol\varepsilon)
=-\frac23\sum_i\varepsilon_i,
\]
the singular points on the special fibres
of $T$ and $Y$ are exactly those for which
$\sum_i\varepsilon_i\equiv0\pmod p$,
with $z=0$ in the case of $Y$.

The action of $G$ identifies sign vectors
whose sums are $ap$ and $-ap$.
The orbits of singularities appearing only
on the special fibres are therefore indexed by
\[
\Theta_p^*
=\{a\in\mathbb Z_{\ge1}:
ap\le k,\ a\equiv k\pmod2\},
\qquad
n_p=\#\Theta_p^*.
\]
When $k$ is even, the balanced sign vectors
($a=0$) are already singular on the generic
fibres. These horizontal singularities
will be treated separately below.

Fix $a\in\Theta_p^*$ and a sign vector
$\boldsymbol\varepsilon=(\varepsilon_1,\ldots,\varepsilon_k)$
with $\sum_i\varepsilon_i=ap$.
Since the Hessian of $g_k$ at
$\boldsymbol\varepsilon$ is nondegenerate,
the \'etale-local Morse lemma gives local equations
\[
Q(u)+c_ap^{v_a}=0 \quad\text{on }T,
\qquad
z^3+Q(u)+c_ap^{v_a}=0 \quad\text{on }Y,
\]
where
\[
Q(u)=-\sum_{i=1}^k\varepsilon_i u_i^2,\qquad
v_a=1+v_p(a),\qquad
c_a=\frac{2a}{3p^{v_p(a)}}\in\mathbb Z_p^\times.
\]
Thus the new special-fibre singularities are
of types $A_1$ and $A_2$, respectively. Hence the
local arithmetic Picard--Lefschetz formulas of
Section~\ref{sec:diag_singularity} apply.

\begin{lemma}[Local acyclicity at horizontal singularities]
\label{lem:airy_horizontal_singularities_acyclic}
At every balanced sign vector ($a=0$), the
vanishing-cycle complexes of $T$ and $Y$ over
$\mathbb Z_p$ vanish. Equivalently, the canonical
maps
\[
\Qlbar\longrightarrow R\Psi\Qlbar
\]
are isomorphisms at these points.
\end{lemma}

\begin{proof}
By the relative Morse lemma, the local equations
at a balanced sign vector are
\[
Q(u)=0 \quad\text{on }T,
\qquad
z^3+Q(u)=0 \quad\text{on }Y,
\]
where $Q(u)=-\sum_i\varepsilon_i u_i^2$
is nondegenerate over $\mathbb Z_p$.

For $T$, blowing up $u=0$ gives a smooth
family with smooth exceptional quadric.
For $Y$, blowing up $z=u=0$ gives a smooth
family whose exceptional divisor is a quadric
cone over a smooth quadric. Blowing up the
singular locus of this cone gives a smooth
family with smooth exceptional divisor.

Applying the blow-up distinguished triangles
first to the quadric cone and then to $Y$,
together with proper and smooth base change,
gives $R\Phi_Y\Qlbar=0$.
For $T$, the corresponding single blow-up
triangle gives $R\Phi_T\Qlbar=0$.
\end{proof}

\begin{lemma}[The sign-isotypic local contributions]
\label{lem:airy_A2_orbit_sign_projector}
For each $a\in\Theta_p^*$, let
$\mathcal O_{a,T}$ and $\mathcal O_{a,Y}$
be the corresponding $G$-orbits of singular
points on the special fibres of $T$ and $Y$.
Put
\[
\begin{aligned}
\mathscr V_a
:=\left(
\bigoplus_{x\in\mathcal O_{a,T}}
(R^{k-1}\Phi_T\Qlbar)_x
\right)^\chi,\qquad
\mathscr W_a
:=\left(
\bigoplus_{x\in\mathcal O_{a,Y}}
(R^k\Phi_Y\Qlbar)_x
\right)^\chi.
\end{aligned}
\]
Then $\mathscr V_a$ has rank one and
$\mathscr W_a$ has rank two. Moreover,
$\mathscr W_a$ contains each nontrivial
geometric $\mu_3$-character exactly once.
\end{lemma}

\begin{proof}
Fix a representative sign vector with
$\sum_i\varepsilon_i=ap>0$.
Its stabilizer in $G$ is
\[
H=S_{(k+ap)/2}\times S_{(k-ap)/2}.
\]
The induced action of $H$ on the local
vanishing cohomology is given by $\chi|_H$.
Indeed, permutations of the quadratic
coordinates act on the top differential form
$du_1\wedge\cdots\wedge du_k$ by their sign,
and the same holds for the local $A_2$
vanishing cohomology.

The vanishing cycles along the orbit form
the representation induced from $H$ to $G$.
Frobenius reciprocity identifies their
$\chi$-isotypic component with a single copy
of the local vanishing-cycle representation.

Theorem~\ref{thm:local_diagonal_main}
then gives $\dim\mathscr V_a=1$ and
$\dim\mathscr W_a=2$, with each nontrivial
geometric $\mu_3$-character occurring
once in $\mathscr W_a$.
\end{proof}

\begin{proposition}[The boundary contribution]
\label{prop:airy_middle_projective_boundary}
Let $e_\chi$ denote the projector onto the
$\chi$-isotypic component, and put
$e=e_\chi e_{\ne1}$.
For $F=\Qp$ or $\Fp$, there is a natural
short exact sequence
\begin{equation}\label{eq:airy_middle_boundary_Amsecond}
0\longrightarrow
H^{k-1}((D_\infty)_{\overline F},\Qlbar)^e
\longrightarrow N_{k,F}
\longrightarrow M^\flat_{k,F}
\longrightarrow0.
\end{equation}
Moreover, specialization induces an isomorphism
\[
H^{k-1}((D_\infty)_{\overline{\mathbb{F}}_{p}},\Qlbar)^e
\xrightarrow{\sim}
H^{k-1}((D_\infty)_{\overline{\mathbb{Q}}_{p}},\Qlbar)^e,
\]
and both spaces have dimension $b_k$.
\end{proposition}

\begin{proof}
Put $V=Y\setminus D_\infty$.
We first compare $V$ with
$A''=V\setminus\{z=0\}$.
Since $\mu_3$ acts trivially on $\{z=0\}$,
localization gives
\[
H_c^k(V_{\overline F})^e
\simeq H_c^k(A''_{\overline F})^e
=N_{k,F}.
\]

For ordinary cohomology, consider the finite
cubic covering $V\to\mathbb A_y^k$ and put
$j:U=\{g_k\ne0\}\hookrightarrow\mathbb A_y^k$.
The associated nontrivial cubic Kummer local
systems $\mathcal K_\rho$ are clean across
$g_k=0$. Indeed, at each quadratic singularity,
blowing up gives a normal-crossings divisor
with multiplicities $1$ and $2$, both having
nontrivial cubic monodromy.
The clean-extension property on this model,
together with proper base change, gives
\[
j_!\mathcal K_\rho\xrightarrow{\sim}
Rj_*\mathcal K_\rho.
\]
Consequently,
\[
H^k(V_{\overline F})^e
\simeq H^k(A''_{\overline F})^e.
\]

Next, consider the localization sequence for
$D_\infty\hookrightarrow Y\hookleftarrow V$.
By the Lefschetz theorem for projective
hypersurfaces with isolated singularities,
$H^{k-1}(Y_{\overline F})^e=0$.
Moreover, weak and hard Lefschetz give
\[
H^{k-2}((D_\infty)_{\overline F})^e
=H^k((D_\infty)_{\overline F})^e=0,
\]
since these groups consist of ambient classes,
on which the sign character does not occur.

Consequently, localization yields
\[
0\longrightarrow H^{k-1}((D_\infty)_{\overline F})^e
\longrightarrow H_c^k(V_{\overline F})^e
\longrightarrow H^k(Y_{\overline F})^e
\longrightarrow0,
\]
while the Gysin sequence gives an injection
\[
H^k(Y_{\overline F})^e
\hookrightarrow H^k(V_{\overline F})^e.
\]
Their composition is the natural map
$H_c^k(V_{\overline F})^e\to H^k(V_{\overline F})^e$.
Hence
\[
H^k(Y_{\overline F})^e
\simeq H_{\mathrm{mid}}^k(V_{\overline F})^e
\simeq M^\flat_{k,F},
\]
which proves \eqref{eq:airy_middle_boundary_Amsecond}.

Finally, since $D_\infty$ is smooth and proper
over $\mathbb Z_p$, smooth proper base change
gives the specialization isomorphism.
By Lemma~\ref{lem:dim_over_Qp_and_de_Rham},
\[
\dim H^{k-1}((D_\infty)_{\overline{\mathbb{Q}}_p})^e
=\dim N_{k,\Qp}-\dim M^\flat_{k,\Qp}
=(d_k+b_k)-d_k=b_k.
\]
The same dimension holds over $\Fp$.
\end{proof}

\begin{theorem}[Specialization for $A'$]
\label{thm:airy_Aprime_specialization}
Suppose either $p>5$ with $k\ge2$, or $p=5$
with $k\ge3$ odd.
Then there is a short exact sequence of local
Galois representations
\begin{equation}\label{eq:airy_Aprime_specialization_SES}
0\longrightarrow H_c^{k-1}(A'_{\bar s})^\chi
\longrightarrow H_c^{k-1}(A'_{\bar\eta})^\chi
\longrightarrow\bigoplus_{a\in\Theta_p^*}\mathscr V_a
\longrightarrow0.
\end{equation}
After a Tate twist by $(-1)$, this gives the corresponding
exact sequence for $C'_{k,F}$, with local terms
$\mathscr V_a(-1)$.
The same short exact sequence holds with
$H_c^{k-1}(A')^\chi$ replaced by
$H_{\mathrm{mid}}^{k-1}(A')^\chi$.

For each $a\in\Theta_p^*$, the local quadratic
Picard--Lefschetz formula gives the inertia
eigenvalue
\[
(-1)^{k(1+v_p(a))}
\]
on $\mathscr V_a$.
\end{theorem}

\begin{proof}
Since $D'$ is smooth over $\mathbb Z_p$,
its vanishing cycles vanish.
By Lemma~\ref{lem:airy_horizontal_singularities_acyclic},
the same holds at the horizontal singularities
when $k$ is even.

The nearby-to-vanishing-cycle triangle
(cf.~\cite[Exp.~XIII, \S2.1]{SGA7_II}),
together with proper base change, therefore
gives the five-term exact sequence
\[
\begin{gathered}
0\longrightarrow H_c^{k-1}(A'_{\bar s})^\chi
\longrightarrow H_c^{k-1}(A'_{\bar\eta})^\chi
\longrightarrow \bigoplus_{a\in\Theta_p^*}\mathscr V_a
\longrightarrow H_c^k(A'_{\bar s})^\chi
\longrightarrow H_c^k(A'_{\bar\eta})^\chi
\longrightarrow0,
\end{gathered}
\]
where the local terms are given by
Lemma~\ref{lem:airy_A2_orbit_sign_projector}.

By Proposition~\ref{prop:Aiprime_cohomology_hypersurface},
\[
H_c^k(A'_{\bar s})^\chi(-1)
\simeq H_c^2(\Gm,\Sym^k\Ai').
\]
The latter vanishes for $p>5$ by
Proposition~\ref{prop:dim_Sym_Ai'_F_p},
and for $p=5$ and odd $k$ by
Proposition~\ref{prop:dim_Aiprime_p5}.
Thus
\[
H_c^k(A'_{\bar s})^\chi
=H_c^k(A'_{\bar\eta})^\chi=0,
\]
and the five-term sequence reduces to
\eqref{eq:airy_Aprime_specialization_SES}.

Finally, the Lefschetz theorems give
\[
H^{k-2}(T_{\overline F})^\chi
=H^{k-3}(D'_{\overline F})^\chi=0.
\]
Hence localization and the Gysin sequence
identify the kernel of
\[
H_c^{k-1}(A'_{\overline F})^\chi
\longrightarrow
H_{\mathrm{mid}}^{k-1}(A'_{\overline F})^\chi
\]
with $H^{k-2}(D'_{\overline F})^\chi$.
Since $D'$ is smooth and proper over
$\mathbb Z_p$, smooth proper base change
identifies these kernels. The snake lemma
then gives the asserted exact sequence for
middle cohomology.
\end{proof}

\begin{remark}\label{rmk:five_term_in_k_even}
For $p=5$ and even $k$, put
\[
h_k=\dim H_c^2(\Gm_{\bar s},\Sym^k\Ai'),
\]
as in Proposition~\ref{prop:dim_Aiprime_p5}.
The dimension formulas of
Propositions~\ref{prop:Aiprime_cohomology_hypersurface},
\ref{prop:dim_Aiprime_p5} and
Lemma~\ref{lem:dim_over_Qp_and_de_Rham}
give
\[
\begin{aligned}
\dim H_c^{k-1}(A'_{\bar\eta})^\chi
-\dim H_c^{k-1}(A'_{\bar s})^\chi
&=n_p-h_k,\\
\dim H_c^k(A'_{\bar s})^\chi
&=h_k.
\end{aligned}
\]
Comparing these dimensions with the $n_p$-dimensional vanishing-cycle term, we obtain
\[
\begin{gathered}
0\longrightarrow H_c^{k-1}(A'_{\bar s})^\chi
\longrightarrow H_c^{k-1}(A'_{\bar\eta})^\chi
\longrightarrow \bigoplus_{a\in\Theta_p^*}\mathscr V_a
\longrightarrow H_c^k(A'_{\bar s})^\chi
\longrightarrow0.
\end{gathered}
\]
In particular, the canonical map to the vanishing-cycle
term has rank $n_p-h_k$, whereas the connecting
morphism has rank $h_k$.
\end{remark}

\begin{theorem}[The characteristic-two comparison for $A'$]
\label{thm:airy_Aprime_p2_comparison}
Let $p=2$, let $k\ge3$ be odd, and let $\ell\ne2$.
Then the specialization map
\[
H_c^{k-1}(A'_{\overline{\mathbb F}_2},\Qlbar)^\chi
\xrightarrow{\sim}
H_c^{k-1}(A'_{\overline{\mathbb Q}_2},\Qlbar)^\chi
\]
is an isomorphism.
\end{theorem}
\begin{proof}
For odd $k$, the generic fibre of $T$ is smooth,
the special fibre has a unique isolated singularity,
and the boundary $D'$ has good reduction.
By the concentration theorem for vanishing
cycles of isolated complete-intersection
singularities (cf.~\cite[\S2.10]{Ill03}),
the specialization map is
injective in degree $k-1$.
By Propositions~\ref{prop:Aiprime_cohomology_hypersurface},
\ref{prop:dim_H1_p2} and
Lemma~\ref{lem:dim_over_Qp_and_de_Rham},
both sides have dimension $(k+1)/2$.
Hence the map is an isomorphism.
\end{proof}

We next study the specialization comparison for $A''$,
focusing on its nontrivial $\mu_3$-part.
We retain $e=e_\chi e_{\ne1}$ and, for
$F\in\{\Qp,\Fp\}$, put
\[
\mathcal C_F=R\Gamma_c(A''_{\overline F},\Qlbar)^e.
\]
The local $A_2$ contributions give a specialization
triangle for these compactly supported realizations.

\begin{proposition}[The specialization triangle]
\label{prop:airy_Nk_specialization_triangle}
For $p\ge5$ and $k\ge2$, there is a distinguished triangle
\begin{equation}\label{eq:airy_Nk_specialization_triangle}
\mathcal C_{\Fp}\longrightarrow\mathcal C_{\Qp}
\longrightarrow
\bigoplus_{a\in\Theta_p^*}\mathscr W_a[-k]
\xrightarrow{+1}.
\end{equation}
\end{proposition}
\begin{proof}
The nearby-to-vanishing-cycle triangle for the proper
family $Y$, together with
Lemmas~\ref{lem:airy_horizontal_singularities_acyclic}
and \ref{lem:airy_A2_orbit_sign_projector}, gives
\[
R\Gamma(Y_{\bar s},R\Phi_Y\Qlbar)^e
\simeq\bigoplus_{a\in\Theta_p^*}\mathscr W_a[-k].
\]
On the other hand, $\mu_3$ acts trivially on $D_0$
and $D_0\cap D_\infty$. Thus Mayer--Vietoris gives
\[
R\Gamma(D_{\overline F},\Qlbar)^e
\simeq R\Gamma((D_\infty)_{\overline F},\Qlbar)^e.
\]
Since $D_\infty$ has good reduction, its
specialization map is an isomorphism. Comparing the
localization triangles
\[
R\Gamma_c(A''_{\overline F})^e\longrightarrow
R\Gamma(Y_{\overline F})^e\longrightarrow
R\Gamma(D_{\overline F})^e\xrightarrow{+1}
\]
for the two fibres proves
\eqref{eq:airy_Nk_specialization_triangle}.
\end{proof}

\begin{theorem}[Arithmetic $A_2$ specialization]
\label{thm:airy_Nk_specialization}
For $p\ge5$ and $k\ge2$, there is a canonical short
exact sequence of local Galois representations
\begin{equation}\label{eq:airy_Nk_specialization_SES}
0\longrightarrow N_{k,\Fp}
\longrightarrow N_{k,\Qp}
\longrightarrow\bigoplus_{a\in\Theta_p^*}\mathscr W_a
\longrightarrow0.
\end{equation}
In particular,
\[
\dim N_{k,\Qp}-\dim N_{k,\Fp}=2n_p.
\]
\end{theorem}
\begin{proof}
Taking cohomology in
\eqref{eq:airy_Nk_specialization_triangle}
gives
\[
\begin{gathered}
0\longrightarrow N_{k,\Fp}
\longrightarrow N_{k,\Qp}
\longrightarrow \bigoplus_{a\in\Theta_p^*}\mathscr W_a
\longrightarrow H_c^{k+1}(A''_{\bar s})^e
\longrightarrow H_c^{k+1}(A''_{\bar\eta})^e
\longrightarrow0.
\end{gathered}
\]
We prove that the map onto the vanishing-cycle
term is surjective using the local intersection
pairing of Theorem~\ref{thm:local_diagonal_main}.

At each new $A_2$ singularity $x$ of $Y_{\bar s}$,
put
\[
V_x=H^k_{\{x\}}(Y_{\bar s},R\Psi\Qlbar),
\qquad
W_x=(R^k\Phi_Y\Qlbar)_x.
\]
The local eigencharacters satisfy
\[
\Lambda_\rho=(-1)^k\rho\ne1,
\qquad \rho\in\mu_3\setminus\{1\}.
\]
By Theorem~\ref{thm:local_diagonal_main},
the local intersection form on $V_x$ is
nondegenerate. Hence the natural map
\[
j_x:V_x\xrightarrow{\sim}W_x
\]
is an isomorphism by local Poincar\'e duality.

Let $q_x:V_x\to H^k(Y_{\bar\eta},\Qlbar)$
be the map forgetting supports.
The compatibility
\[
\operatorname{can}_y\circ q_x
=\begin{cases}
j_x,&x=y,\\
0,&x\ne y
\end{cases}
\]
shows that
\[
H^k(Y_{\bar\eta},\Qlbar)^e
\twoheadrightarrow
\bigoplus_{a\in\Theta_p^*}\mathscr W_a.
\]

On the other hand, localization gives
\[
H_c^k(A''_{\bar\eta},\Qlbar)^e
\twoheadrightarrow H^k(Y_{\bar\eta},\Qlbar)^e,
\]
since
\[
H^k(D_{\bar\eta},\Qlbar)^e
\simeq H^k((D_\infty)_{\bar\eta},\Qlbar)^e
=0.
\]
Consequently,
\[
N_{k,\Qp}\twoheadrightarrow
\bigoplus_{a\in\Theta_p^*}\mathscr W_a,
\]
and the five-term sequence reduces to
\eqref{eq:airy_Nk_specialization_SES}.

Finally, since each $\mathscr W_a$ has rank two,
\[
\dim N_{k,\Qp}-\dim N_{k,\Fp}
=2|\Theta_p^*|=2n_p.
\]
\end{proof}

\begin{theorem}[Specialization of the middle geometric factor]
\label{thm:airy_middle_Nk_specialization}
For $p\ge5$ and $k\ge2$, there is a canonical
short exact sequence of local Galois representations
\begin{equation}\label{eq:airy_middle_specialization_SES}
0\longrightarrow M^{\flat}_{k,\Fp}
\longrightarrow M^{\flat}_{k,\Qp}
\longrightarrow\bigoplus_{a\in\Theta_p^*}\mathscr W_a
\longrightarrow0.
\end{equation}
Consequently,
\[
\dim M^{\flat}_{k,\Qp}-\dim M^{\flat}_{k,\Fp}=2n_p.
\]
\end{theorem}
\begin{proof}
Put
$B_{k,F}=H^{k-1}((D_\infty)_{\overline F},\Qlbar)^e$.
Proposition~\ref{prop:airy_middle_projective_boundary}
gives the commutative diagram
\[
\begin{array}{ccccccccc}
0&\longrightarrow&B_{k,\Fp}&\longrightarrow&N_{k,\Fp}
 &\longrightarrow&M^\flat_{k,\Fp}&\longrightarrow&0\\
 &&\downarrow\wr&&\downarrow&&\downarrow&&\\
0&\longrightarrow&B_{k,\Qp}&\longrightarrow&N_{k,\Qp}
 &\longrightarrow&M^\flat_{k,\Qp}&\longrightarrow&0.
\end{array}
\]
The left-hand map is an isomorphism by smooth proper
base change, while the middle map is injective with
cokernel $\bigoplus_a\mathscr W_a$ by
Theorem~\ref{thm:airy_Nk_specialization}.
The snake lemma gives \eqref{eq:airy_middle_specialization_SES}.
\end{proof}

\begin{remark}
The specialization sequence
\eqref{eq:airy_middle_specialization_SES}
can also be obtained by a dimension argument.
By Proposition~\ref{prop:G_geom_Ai_Ai'},
the geometric monodromy groups for $p\ge5$ are
perfect. Thus the nontrivial cubic Kummer twists
of $\Sym^k\Ai'$ have no geometrically constant
summands, so their compactly supported cohomology
vanishes in degree two. The cubic-cover
decomposition and the Grothendieck--Ogg--Shafarevich
formula of \S\ref{subsec:airy_sheaves} therefore give
\[
\dim N_{k,\Fp}
=2\Swan_\infty(\Sym^k\Ai')
=2\left\lfloor\frac{k+1}{2}\right\rfloor-2n_p.
\]
Combining this with Lemma~\ref{lem:dim_over_Qp_and_de_Rham},
we obtain
\[
\dim N_{k,\Qp}-\dim N_{k,\Fp}=2n_p.
\]

On the other hand, by
Lemma~\ref{lem:airy_A2_orbit_sign_projector},
\[
\dim\bigoplus_{a\in\Theta_p^*}\mathscr W_a=2n_p.
\]
Thus the five-term exact sequence associated
with \eqref{eq:airy_Nk_specialization_triangle}
shows that
\[
N_{k,\Qp}\longrightarrow
\bigoplus_{a\in\Theta_p^*}\mathscr W_a
\]
is surjective. The boundary comparison of
Proposition~\ref{prop:airy_middle_projective_boundary}
then gives \eqref{eq:airy_middle_specialization_SES}.
\end{remark}

\paragraph{Inertia and Frobenius.}
Let $I_p\subseteq\Gal(\overline{\mathbb{Q}}_p/\Qp)$ be inertia.
By Theorem~\ref{thm:local_diagonal_main}, tame inertia
on $\mathscr W_a$ has eigenvalues
\[
\bigl((-1)^k\zeta_3\bigr)^{v_a},\qquad
\bigl((-1)^k\zeta_3^2\bigr)^{v_a}.
\]
In particular,
\begin{equation}\label{eq:inertia_inv_cond}
\mathscr W_a^{I_p}=
\begin{cases}
\mathscr W_a,&6\mid v_a\text{ if }k\text{ is odd},\\
\mathscr W_a,&3\mid v_a\text{ if }k\text{ is even},\\
0,&\text{otherwise}.
\end{cases}
\end{equation}
The local Frobenius action on these invariant terms
is determined by Theorem~\ref{thm:Frob_trace_on_inertia_inv}.

\begin{theorem}[Inertia invariants and Frobenius]
\label{thm:airy_inertia_connecting_zero}
Taking inertia invariants in
\eqref{eq:airy_middle_specialization_SES} gives a
short exact sequence
\begin{equation}\label{eq:airy_inertia_invariant_SES}
0\longrightarrow M^{\flat}_{k,\Fp}
\longrightarrow (M^{\flat}_{k,\Qp})^{I_p}
\longrightarrow\bigoplus_{a\in\Theta_p^*}\mathscr W_a^{I_p}
\longrightarrow0.
\end{equation}
Consequently, with $P_V(T)=\det(1-T\Frob_p\mid V)$,
\[
P_{(M^{\flat}_{k,\Qp})^{I_p}}(T)
=P_{M^{\flat}_{k,\Fp}}(T)
\prod_{a\in\Theta_p^*}P_{\mathscr W_a^{I_p}}(T).
\]
The same assertions hold for the compactly supported cohomology sequence \eqref{eq:airy_Nk_specialization_SES}, with
$M^{\flat}_{k,F}$ replaced by $N_{k,F}$.
\end{theorem}
\begin{proof}
Taking inertia invariants in
\eqref{eq:airy_middle_specialization_SES} gives
\[
\begin{aligned}
0\longrightarrow M^{\flat}_{k,\Fp}
&\longrightarrow (M^{\flat}_{k,\Qp})^{I_p}
\longrightarrow\bigoplus_a\mathscr W_a^{I_p}\xrightarrow{\partial}H^1(I_p,M^{\flat}_{k,\Fp}).
\end{aligned}
\]
Since $A''_{\bar s}$ is smooth, the middle cohomology
$M^{\flat}_{k,\Fp}$ is pure of weight $k$, and
\[
H^1(I_p,M^{\flat}_{k,\Fp})
\simeq M^{\flat}_{k,\Fp}(-1)
\quad\text{has weight }k+2.
\]
Each nonzero $\mathscr W_a^{I_p}$ is pure of weight $k$
by the local Gauss--Jacobi sum formula. Since $\partial$ is
Frobenius-equivariant, the weight difference forces
$\partial=0$, giving \eqref{eq:airy_inertia_invariant_SES}.

For the compactly supported cohomology sequence, the boundary exact
sequence of
Proposition~\ref{prop:airy_middle_projective_boundary}
shows that $N_{k,\Fp}$ has weights among $k-1$ and $k$.
Thus $H^1(I_p,N_{k,\Fp})$ has weights $k+1$
and $k+2$, so its connecting map also vanishes.
The characteristic-polynomial identities follow
from the short exact sequences.
\end{proof}

\begin{theorem}[Local Frobenius factors]
\label{thm:airy_local_frobenius_factors}
Let $a\in\Theta_p^*$.
If $6\nmid v_a$ for odd $k$, or $3\nmid v_a$
for even $k$, then
\[
P_{\mathscr W_a^{I_p}}(T)=1.
\]
Otherwise,
\[
P_{\mathscr W_a^{I_p}}(T)
=\begin{cases}
1-(\Lambda_a+\overline\Lambda_a)T+p^kT^2,
&p\equiv1\pmod3,\\
1+(-1)^{k+1}p^kT^2,
&p\equiv2\pmod3.
\end{cases}
\]
Here $\chi_2$ denotes the quadratic character of
$\Fp^\times$. For $p\equiv1\pmod3$, fix a cubic character
$\chi_3$ of $\Fp^\times$ and put
\[
\Lambda_a=
\chi_3(c_a)\chi_2(c_a)^k
\chi_2(-1)^{(k+ap)/2}J_1(\chi_3,\underbrace{\chi_2,\ldots,\chi_2}_{k}).
\]
Let $\overline{\Lambda}_a$ denote the expression
obtained by replacing $\chi_3$ with $\chi_3^{-1}$.

\end{theorem}

\begin{proof}
The inertia criterion follows from
\eqref{eq:inertia_inv_cond}.
Suppose $\mathscr W_a^{I_p}=\mathscr W_a$.
By Theorem~\ref{thm:local_diagonal_main}
and Proposition~\ref{prop:local_duality},
$\mathscr W_a$ carries a nondegenerate
$(-1)^k$-symmetric pairing
\[
B_a:\mathscr W_a\times\mathscr W_a
\longrightarrow\Qlbar(-k),
\]
satisfying
\[
B_a(\Frob_p x,\Frob_p y)=p^k B_a(x,y).
\]

Write $\mathscr W_a=L_\rho\oplus L_{\rho^{-1}}$
for its two geometric cubic character lines.
The pairing identifies these lines as duals.

If $p\equiv1\pmod3$, Frobenius preserves
both lines, so the pairing gives
\[
\det(\Frob_p\mid\mathscr W_a)=p^k.
\]
The trace $\Lambda_a+\overline\Lambda_a$
follows from the characterwise Gauss--Jacobi
formula \eqref{eq:individual_diagonal_Frob_trace}.

If $p\equiv2\pmod3$, Frobenius exchanges
the two lines, so its trace is zero.
Writing
\[
\Frob_p e_\rho=u e_{\rho^{-1}},
\qquad
\Frob_p e_{\rho^{-1}}=v e_\rho,
\]
the symmetry of $B_a$ gives
\[
uv=(-1)^kp^k.
\]
Hence
\[
\det(\Frob_p\mid\mathscr W_a)
=-uv=(-1)^{k+1}p^k,
\]
which completes the proof.
\end{proof}

\bibliographystyle{alpha}
\bibliography{cite}

\end{document}